\documentclass[12pt]{amsart}
\usepackage{fullpage}

\usepackage{amssymb,amsfonts,amsmath}

\usepackage{graphicx}
\usepackage{amssymb}
\usepackage{epstopdf}

\usepackage[all,cmtip]{xy}

\usepackage[T1]{fontenc}
\usepackage[utf8]{inputenc}
\usepackage{mathtools} 

\usepackage{subfig}

\usepackage[bbgreekl]{mathbbol}

\usepackage{tikz}

\usepackage{color}

\renewcommand{\int}{\operatorname{int}} 
 
\newcommand{\Ker}{\operatorname{Ker}}

\newcommand{\Arf}{\mathrm{Arf}}

\newcommand\W{{\sf W}} 
 
\newcommand\sL{{\sf L}}
\newcommand\sD{{\sf D}}
\newcommand\sK{{\sf K}}

\newcommand{\N}{\mathbb{N}} 

\newcommand{\bW}{\mathbb{W}}

\newcommand{\cG}{\mathcal{G}}

\newcommand{\cT}{\mathcal{T}} 
 
\newcommand{\cW}{\mathcal{W}}

\newcommand{\hra}{\hookrightarrow} 
\newcommand{\imra}{\looparrowright} 
\newcommand{\sra}{\twoheadrightarrow}

\newcommand{\iinfty}{{\mathchoice
{\begin{minipage}{.15in}\includegraphics[width=.12in]{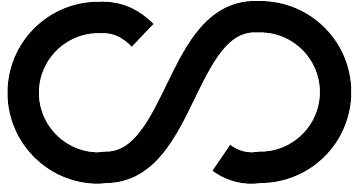}\end{minipage}}
{\begin{minipage}{.10in}\includegraphics[width=.10in]{Figures/infty2.pdf}\end{minipage}}
{\begin{minipage}{.08in}\includegraphics[width=.08in]{Figures/infty2.pdf}\end{minipage}}
{\begin{minipage}{.08in}\includegraphics[width=.08in]{Figures/infty2.pdf}\end{minipage}}
}}

\usepackage{amsthm}

\newtheorem{thm}{Theorem}[section]
\newtheorem{lem}[thm]{Lemma}

\newtheorem{defn}[thm]{Definition}
\newtheorem{rem}[thm]{Remark}
\newtheorem{cor}[thm]{Corollary}
\newtheorem{conj}[thm]{Conjecture}
\newtheorem{prop}[thm]{Proposition}

\newcommand{\introthmname}{}
\newtheorem{introthminn}{\introthmname}
\newenvironment{introthm}[1]
  {\renewcommand{\introthmname}{#1}\begin{introthminn}}
  {\end{introthminn}}

\swapnumbers

\theoremstyle{plain}
\theoremstyle{definition}
\newcommand{\Z}{\mathbb{Z}}
\newcommand{\cC}{\mathcal{C}}
\newcommand{\cI}{\mathcal{I}}
\title{Clasper Concordance, Whitney towers \\ and repeating Milnor invariants}
\author[J. Conant]{James Conant}
\email{jim.conant@gmail.com}
\address{GIA, Carlsbad, CA}

\author[R. Schneiderman]{Rob Schneiderman}
\email{robert.schneiderman@lehman.cuny.edu}
\address{Dept. of Mathematics, Lehman College, City University of New York, Bronx, NY}

\author[P. Teichner]{Peter Teichner}
 \email{teichner@mac.com}
\address{Max-Planck-Institut f\"ur Mathematik, Bonn, Germany}

\begin{document}
\maketitle

\begin{abstract}
We show that for each $k\in \N$, a link $L\subset S^3$ bounds a degree $k$ Whitney tower in the 4-ball if and only if it is \emph{$C_k$-concordant} to the unlink. This means that $L$ is obtained from the unlink by a finite sequence of concordances and degree $k$ clasper surgeries. In our construction the trees associated to the Whitney towers coincide with the trees associated to the claspers.

As a corollary to our previous obstruction theory for Whitney towers in the 4-ball, it follows that the $C_k$-concordance filtration of links is classified in terms of Milnor invariants, higher-order Sato-Levine and Arf invariants.

Using a new notion of $k$-repeating twisted Whitney towers, we also classify a natural generalization of the notion of link homotopy, called twisted \emph{self $C_k$-concordance}, in terms of $k$-repeating Milnor invariants and $k$-repeating Arf invariants. 
\end{abstract}

%\keywords{$C_k$-equivalence, $C_k$-concordance, self $C_k$-concordance, $k$-repeating Milnor invariants, Whitney towers, claspers}

%\ccode{Mathematics Subject Classification 2000: 57M25, 57M27}

\section*{Introduction}\label{sec:intro}

Two links $L,L'\subset S^3$ are \emph{clasper concordant} if $L$ is related to $L'$ by a finite sequence of concordances (in $S^3 \times [0,1]$) and clasper surgeries (in $S^3$). For the latter we only allow {\em simple} claspers associated to {\em trivalent trees}, see Section~\ref{sec:claspers}. Recall also that a link $L\subset S^3$ \emph{bounds a Whitney tower} if the components of $L$ bound generic disks in $B^4$ together with iterated stages of Whitney disks, pairing the arising intersections. The unpaired intersections in the top layers of the Whitney tower are again organized by trivalent trees, see Section~\ref{sec:w-towers}. 

\begin{introthm}{Theorem}\label{thm:clasper-concordance}

A link $L\subset S^3$ bounds a Whitney tower $\cW$ with associated trees $t(\cW)$ if and only if $L$ is clasper concordant to the unlink with $t(\cC)=t(\cW)$, where $t(\cC)$ denotes the trees associated to the clasper surgeries $\cC=(C_1,\dots,C_r)$ that are part of the clasper concordance. 

\end{introthm}
The ``only if'' direction of this theorem is new and was previously announced and (used!) in \cite[Thm. 3.1]{CST7}. It will follow from the more precise Theorems~\ref{thm:w-tower-forest-implies-clasper-forest} and~\ref{thm:clasper-forest-implies-w-tower-forest}.
 
The effect of Goussarov--Habiro clasper surgery \cite{GGP}, \cite[Fig.34]{Hab00} on links corresponds to ``iterated Bing doubling along a tree'' introduced by Tim Cochran \cite{C} to realize Milnor invariants. We have  shown in \cite{CST,CST1} that if $L$ is the result of clasper surgery on the unlink along a collection $\cC$ of tree claspers, then $L$ bounds a Whitney tower $\cW$ into the $4$--ball such that the associated collections of trees $t(\cW)$ and $t(\cC)$ are identical (Figure~\ref{3-component-link-clasper-tower-fig}), implying the ``if'' direction of Theorem~\ref{thm:clasper-concordance}. The proof of the new ``only if'' direction involves decomposing a given $\cW$ into local Whitney towers around each tree in $t(\cW)$ which are bounded by iterated Bing doubles and can be expressed as clasper surgeries and concordances. 

\begin{figure}[ht!]
        \centerline{\includegraphics[scale=.35]{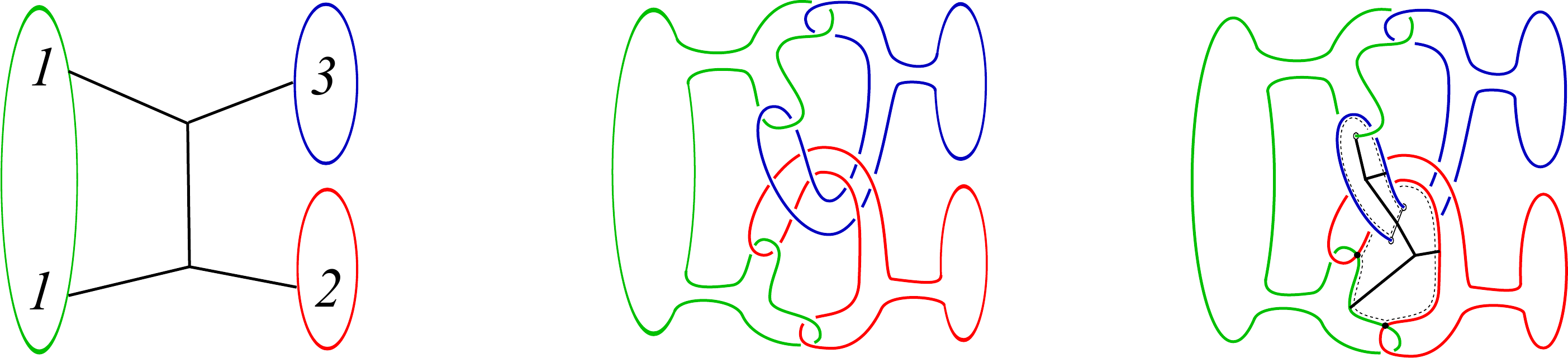}}
        \caption{Left: A tree $^3_1>\!\!\!-\!\!\!\!-\!\!\!\!\!-\!\!\!<^{\,2}_{\,1}\,=t(C)$ associated to a clasper $C$ on the $3$-component unlink. Center: The link $L$ resulting from surgery along $C$. Right: A Whitney tower $\cW$ bounded by $L$ with $t(\cW)=\,^3_1>\!\!\!-\!\!\!\!-\!\!\!\!\!-\!\!\!<^{\,2}_{\,1}$ can be constructed from the track of a null-homotopy of $L$ which changes two red-green crossings.}
        \label{3-component-link-clasper-tower-fig}
\end{figure}

Two links are \emph{$C_k$-equivalent} (cf.~Habiro--Goussarov) if they are connected by a finite sequence of simple tree clasper surgeries of degree $\geq k$, where the \emph{degree} of a clasper is half the number of vertices of the associated tree.
It has been conjectured by Habiro and Goussarov that finite type $\leq k$ invariants classify string links up to $C_k$-equivalence, as has been proved for knots \cite{Gu,Hab00}. 

Two links are \emph{$C_k$-concordant} if they are connected by a finite sequence of $C_k$-equivalences and concordances. In other words, the links are clasper concordant in a way that each arising clasper has degree $\geq k$. This notion was introduced by Meilhan and Yasuhara in \cite{MY}.

Similarly, a Whitney tower has degree $k$ if all its associated trees are of degree $\geq k$, see Definition~\ref{def:w-tower-order}. The following result holds for all $k\in\N$.

\begin{introthm}{Corollary}\label{cor:order-n-w-tower-equals-Ck-conc}

A link is $C_k$-concordant to the unlink iff it bounds a  Whitney tower of degree~$k$.%$\hfill\square$
%\end{cor}
\end{introthm}
It remains an open problem whether two links $L$ and $L'$ are $C_k$-concordant provided they cobound a Whitney tower concordance of degree $k$. Our argument requires the unlink on one side of the Whitney concordance.

To connect more directly to our previous results, we will now switch to \emph{order} for measuring the complexity of a Whitney tower (section~\ref{subsec:w-tower-order}).  The order of a tree is the number of trivalent vertices, hence order is just one less than degree. For example, if a knot $K\subset S^3$ bounds an immersed disk $\Delta\imra B^4$ together with framed, immersed Whitney disks for all self-intersections of $\Delta$, these data constitute a Whitney tower of order 1, or degree 2, and the associated trees would correspond to interior intersections between the Whitney disks and $\Delta$, with each tree having the shape of a Y. Such a Whitney tower exists if and only if Arf$(K)=0$.

In \cite[Thm.1.1]{CST1} we showed that a link $L$ bounds an order $n$ Whitney tower if and only if $L$ has all vanishing Milnor invariants, higher-order Sato--Levine invariants and higher-order Arf invariants through order $n$. Here order is two less than \emph{length} in traditional Milnor invariant terminology (see section~\ref{subsec:intro-Milnor-review}), and the new Sato--Levine invariants appear in all odd orders %$2j-1$
as certain projections of Milnor invariants in the next higher even order % $2j$
\cite[Sec.5]{CST1}.
A version of the higher-order Arf invariants will be defined in Section~\ref{subsec:intro-higher-order-arf}.

Combining \cite[Thm.1.1]{CST1} with Corollary~\ref{cor:order-n-w-tower-equals-Ck-conc}, we get the following algebraic characterization of $C_{k}$-concordance to the unlink:

\begin{introthm}{Corollary}\label{cor:Milnor-SL-Arf-equals-tower-equals-Ck}
For any $k\in\N$ and any link $L\subset S^3$, the following statements are equivalent: 
\begin{enumerate}
\item
$L$ is $C_{k}$-concordant to the unlink.
\item
$L$ bounds a Whitney tower of order $(k-1)$ (aka degree $k$).
\item
$L$ has vanishing Milnor, Sato--Levine and Arf invariants in orders $< k$.
\end{enumerate}%$\hfill\square$
%\end{cor}
\end{introthm}

In fact, the results of \cite{CST1,CST2} led to the discovery that the notion of order $k$ \emph{twisted} Whitney towers provided a complexity that better expressed the geometry of Milnor invariants, and in particular the higher-order Sato-Levine invariants are not needed.
As detailed in Section~\ref{sec:w-towers}, twisted Whitney towers of order $k$ are allowed to contain certain Whitney disks of order $\geq k/2$ which are not framed,
and in Section~\ref{sec:claspers} we define analogous \emph{twisted claspers} along with a corresponding notion of \emph{twisted $C_k$-concordance}. The leaves of a {\em twisted} clasper on a link $L$ are all zero-framed meridians to the components of $L$ except for one leaf which bounds an embedded disk in the complement of $L$ which has non-zero framing. The special $\iinfty$-trees associated to twisted Whitney disks and twisted claspers are also covered by
Theorem~\ref{thm:clasper-concordance}, as explained in Section~\ref{sec:Ck-conc-and-w-towers}.  
Combining Theorem~\ref{thm:clasper-concordance} with \cite[Cor.1.16]{CST1}, we have:

\begin{introthm}{Corollary}\label{cor:Milnor-Arf-equals-twisted-tower-equals-twisted-Ck}
For any $k\in\N$ and any link $L\subset S^3$, the following statements are equivalent: 
\begin{enumerate}
\item
$L$ is twisted $C_{k}$-concordant to the unlink.
\item
$L$ bounds a twisted Whitney tower of order $(k-1)$.
\item
$L$ has vanishing Milnor and Arf invariants in orders $< k$.
\end{enumerate}
\end{introthm}

%We will see next that twisted Whitney towers also provide a nice correspondence with natural variations of Milnor invariants and $C_k$-concordance, namely the notions of $k$-repeating Milnor invariants and twisted self-$C_k$ concordance.

\textbf{Twisted self $C_k$-concordance, $k$-repeating Milnor invariants.}
A \emph{self $C_k$-equivalence} is a $C_k$-equivalence which only allows surgery on claspers whose leaves are all meridians to the \emph{same} link component \cite{SY}. That is, each tree $t$ associated to a clasper in a self $C_k$-equivalence is \emph{mono-labeled}, meaning that all univalent vertices of $t$ have the same label (but two trees associated to two different claspers in the same $C_k$-equivalence can each be mono-labeled by a different label). 

A \emph{self $C_k$-concordance} is a finite sequence of self $C_k$-equivalences and concordances. 
Self $C_k$-equivalence and self $C_k$-concordance both generalize Milnor's notion of \emph{link homotopy}. A link homotopy is realized by a sequence of self-crossing changes which is exactly a self $C_1$-equivalence. Moreover, since concordance implies link homotopy \cite{Gi,Go} we see that the notions of self $C_1$-equivalence, self $C_1$-concordance and link homotopy all coincide. 

 Recall that for an $m$-component link $L$, the multi-indices $\cI$ which parametrize the Milnor concordance invariants $\mu_\cI(L)$ take their entries from the set $\{1,2,\ldots,m\}$ indexing the components of $L$. For a multi-index $\cI$, denote by $r_i(\cI)$ the number of occurrences of the label $i$ in $\cI$, and denote by $r(\cI)$ the maximum of $r_i(\cI)$ over the index set. Then $\mu_\cI(L)$ is a \emph{$k$-repeating Milnor invariant} if $r(\cI)=k$. The case $k=1$ recovers the original ``non-repeating'' link-homotopy $\mu$-invariants, with each label occurring at most once in $\cI$. 

The $k$-repeating $\mu_\cI$ are invariants of self $C_k$-equivalence by \cite{FY}, and hence also of self $C_k$-concordance. Milnor showed that a link is self $C_1$-equivalent (link homotopic) to the unlink if and only if all of its $1$-repeating $\mu$-invariants vanish \cite{M1}, and Yasuhara showed that a link is self $C_2$-equivalent (also called \emph{self $\Delta$-equivalent}) to the unlink if and only if all of its $1$-repeating and $2$-repeating $\mu$-invariants vanish \cite{Ya1}.

Understanding self $C_k$-equivalence in general is a difficult problem, due to the multitude of finite type isotopy invariants which are not Milnor invariants, but it is natural to ask about the role of $k$-repeating Milnor invariants in characterizing self $C_k$-concordance.

Self $C_2$-concordance was studied in \cite{Shib1,Shib2} (where it was called $\Delta$-cobordism), and the classification of string links up to self $C_2$-concordance was given by Yasuhara in \cite{Ya2}. In \cite[Thm.6.1]{MY} Meilhan and Yasuhara classified self $C_3$-concordance of $2$-component string links in terms of the classical Arf invariants of the components, the $k$-repeating $\mu$-invariants for $k\leq 3$, and the reductions modulo $2$ of the $4$-repeating $\mu$-invariants (which are an example of the higher-order Sato-Levine invariants from Corollary~\ref{cor:Milnor-SL-Arf-equals-tower-equals-Ck}).
However, in the setting of \emph{twisted} self $C_k$-concordance (twisted $C_k$-concordance restricted to mono-labeled claspers, see section~\ref{subsec:twisted-(self)-Ck-conc-def}) the Sato-Levine invariants do not appear, and from Theorem~\ref{thm:clasper-concordance}, together with a $k$-repeating analogue of Corollary~\ref{cor:Milnor-Arf-equals-twisted-tower-equals-twisted-Ck} in the form of Corollary~\ref{cor:k-repeating-mu-arf-classify-k-repeating-twisted}, we have for any $k\in \N$:

\begin{introthm}{Theorem}\label{thm:twisted-self-Ck-equals-Milnor-Arf-equals-tower}
For an $m$-component link $L$ the following statements are equivalent: 
\begin{enumerate}
\item
$L$ is twisted self $C_k$-concordant to the unlink.
\item
$L$ bounds a $k$-repeating twisted Whitney tower of degree $m\cdot k$.
\item
For each $r$, $1\leq r \leq k$, $L$ has vanishing $r$-repeating Milnor  and
Arf invariants.
\end{enumerate}

\end{introthm}
Theorem~\ref{thm:twisted-self-Ck-equals-Milnor-Arf-equals-tower} is proved in Section~\ref{sec:proof-thm:twisted-self-Ck-equals-Milnor-Arf-equals-tower}, where we introduce twisted $k$-repeating Whitney towers and show that the associated order-raising obstruction theory corresponds to the $k$-repeating Milnor invariants and  $k$-repeating Arf invariants. Note that implicitly this order is bounded, which is easiest expressed in terms of length: If each of the $m$ indices repeats exactly $k$ times then the length of the last possibly non-trivial invariant is $m \cdot k$, or equivalently of degree $m \cdot k-1$. This is also the last invariant that vanishes for links which bound $k$-repeating twisted Whitney towers of degree $m \cdot k$.\vspace{4mm}

{\bf Acknowledgments:} This paper was partially written while the first two authors were visiting the third author at the Max-Planck-Institut f\"ur Mathematik in Bonn. They all thank MPIM for its stimulating research environment. The second author was supported by a Simons Foundation \emph{Collaboration Grant for Mathematicians}.

\tableofcontents

%%%%%%%%%%%%%%%%%%%%%%%%%%%%%%%%%%%%%%%%%%%%%%%%%
%%%%%%%%%%%%%%%%%%%%%%%%%%%%%%%%%%%%%%%%%%%%%
\section{Trees}\label{sec:trees}

This section fixes notation and terminology for the decorated trees that will be associated to
Whitney towers, claspers, Lie brackets and iterated commutators.

In this paper a \emph{tree} is a finite unitrivalent tree, equipped with cyclic orientations at trivalent vertices, and with each univalent vertex (usually)
labeled by an element from the index set $\{1,2,3,\ldots,m\}$. A \emph{rooted tree} has a designated preferred univalent vertex called the \emph{root}. The root of a rooted tree is usually indicated by being the only univalent vertex which is not labeled by an element of the index set.
A \emph{twisted tree}, or $\iinfty$-tree, has one univalent vertex labeled by the \emph{twist} symbol~``$\iinfty$'' and all other univalent vertices labeled from from the index set (so
a twisted tree is gotten by labeling the root of a rooted tree by $\iinfty$).

When the context is clear, the word ``trees'' may refer to all these types of trees.

The adjectives ``un-rooted'', or ``non-$\iinfty$'', may also be occasionally applied to trees for emphasis.

Un-rooted non-$\iinfty$ trees are sometimes referred to as \emph{framed} trees.

(The terminology ``twisted'' and ``framed'' will be explained by the association of trees with Whitney disks and their intersections -- see the paragraph after Definition~\ref{def:split-w-tower}.)

The following definition applies to all three of the above types of trees: framed, twisted, and rooted.
\begin{defn}\label{def:tree-order}
The \emph{order} of a tree is defined to be the number of trivalent vertices.
\end{defn}\label{def:tree-order}

We adopt the convention that the trivalent orientations of trees illustrated in figures are induced by a fixed orientation of the plane.

\begin{defn}\label{def:Tree-ops}
Let $I$ and $J$ be two rooted trees.
\begin{enumerate} 
\item\label{def-item:rooted-product}
 The \emph{rooted product} $(I,J)$ is the rooted tree gotten
by identifying the root vertices of $I$ and $J$ to a single vertex $v$ and sprouting a new rooted edge at $v$.
This operation corresponds to the formal (non-associative) bracket of elements from the index set (Figure~\ref{inner-product-trees-fig} upper right). 

\begin{figure}[ht!]
        \centerline{\includegraphics[scale=.40]{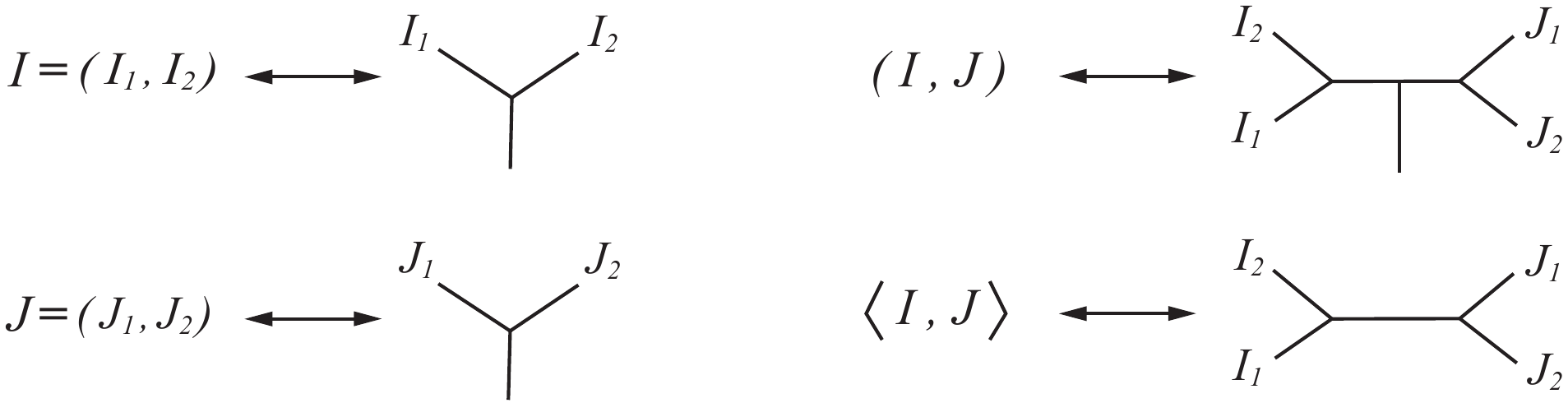}}
        \caption{The \emph{rooted product} $(I,J)$ and \emph{inner product} $\langle I,J \rangle$ of rooted trees $I=(I_1,I_2)$ and $J=(J_1,J_2)$.
        Here $I_1,I_2,J_1,J_2$ each represent rooted trees which are contained as subtrees of $I$ and $J$.}
        \label{inner-product-trees-fig}
\end{figure}
\item\label{def-item:inner-product}
 The \emph{inner product}  $\langle I,J \rangle $ is the
unrooted tree gotten by identifying the roots of $I$ and $J$ to a single non-vertex point.
(Figure~\ref{inner-product-trees-fig} lower right.)

\end{enumerate}
\end{defn}

Other than drawing explicit pictures of trees, we have the following descriptive conventions: 
Rooted trees are usually denoted by upper-case letters, and indexed by (identified with) formal non-associative bracketings of elements from
the index set as in item~(\ref{def-item:rooted-product}) of Definition~\ref{def:Tree-ops}. 
Twisted trees are denoted by adding a $\iinfty$-superscript to a bracketing for the corresponding rooted tree, for instance $J^\iinfty$.
Framed trees are usually denoted by lower case letters, especially ``$t$'', or by the inner product as in item~(\ref{def-item:inner-product}) of Definition~\ref{def:Tree-ops}.

%%=============================================================
%%%%%%%%%%%%%%%%%%%%%%%%%%%%%%%%%%%%%%%%%%

\section{Whitney towers}\label{sec:w-towers}
This section sketches some basic background material on (twisted) Whitney towers.
Additional Whitney tower material relevant to the proof of Theorem~\ref{thm:twisted-self-Ck-equals-Milnor-Arf-equals-tower} will be sketched in Section~\ref{sec:proof-thm:twisted-self-Ck-equals-Milnor-Arf-equals-tower}. 
See e.g.~\cite{CST1} for more details.

\subsection{Whitney disks and their twistings}\label{subsec:twisted-w-disks}

Let $A$ and $B$ be oriented connected immersed surfaces in an oriented $4$--manifold $X$, and
suppose there exists a pair of oppositely signed transverse intersection points $p$ and $q$ in $A\cap B$. 
Joining these points by any embedded pair of arcs lying in the interiors of $A$ and $B$ and avoiding any other intersection points forms a loop, and any generically immersed disk $W$ in (the interior of) $X$ bounded by such a loop is called a \emph{Whitney disk}.

The normal disk-bundle of a Whitney disk $W$ in $X$ is isomorphic to $D^2\times D^2$,
and comes equipped with a nowhere-vanishing \emph{Whitney section} over the boundary $\partial W$ given by pushing $\partial W$  tangentially along one surface and normally along the other, avoiding the tangential direction of $W$.

The
relative Euler number $\omega(W)\in\Z$ of the Whitney section represents the obstruction to extending
the Whitney section to a nowhere-vanishing section over $W$. Following traditional terminology, when $\omega(W)$ vanishes $W$ is said to be \emph{framed}. 
The value of $\omega(W)$ is called the \emph{twisting} of $W$,
and when $\omega(W)$ is non-zero we say that $W$ is \emph{twisted}.

%====================================

\subsection{Definition of Whitney towers}\label{subsec:w-towers-def}

\begin{defn}\label{def:w-tower}
Let $A$ be an oriented properly immersed surface in an oriented $4$--manifold. 
A \emph{Whitney tower} supported by $A$ is defined by:
\begin{enumerate}
\item
$A$ itself is a Whitney tower.
\item
If $\cW$ is a Whitney tower and $W$ is a Whitney disk pairing intersections in $\cW$, then
the union $\cW\cup W$ is a Whitney tower.
\end{enumerate}
\end{defn}
If $\cW$ is a Whitney tower supported by $A$, then we also say that $\cW$ is a Whitney tower ``on'' $A$.

We say that a link $L\subset S^3$ ``bounds a Whitney tower $\cW$'' if the components of $L$ bound properly immersed disks into $B^4$ which support $\cW$.  And a \emph{Whitney tower concordance} between two links is a Whitney tower on immersed annuli in $S^3\times [0,1]$ cobounded by the links in $S^3\times \{0,1\}$.

If a Whitney tower $\cW$ on $A$ contains no unpaired intersections, and if all the Whitney disks in $\cW$ are framed, then $A$ is regularly homotopic to an embedding by performing Whitney moves along the Whitney disks in $\cW$. So constructing a Whitney tower on $A$ can be thought of as an attempt to ``approximate'' a homotopy to an embedding.
On the other hand, the unpaired intersections and Whitney disk twistings in $\cW$ can represent obstructions to the existence of a homotopy of $A$ to an embedding.

\subsection{Trees in Whitney towers}\label{subsec:trees-in-towers}
The following definition describes the association of trees to Whitney disks and unpaired intersections in a Whitney tower.
As illustrated in Figure~\ref{WdiskIJandIJKint-fig} these trees can be considered to be immersed in the Whitney tower,
and
we adopt the convention that tree orientations correspond to Whitney disk orientations via this point of view.
However, specific orientation choices and conventions will usually be suppressed from notation/discussion.

%%%%%%%%%%%%%%%%%%%%%%%%%%%%
\begin{defn}\label{def:int-and-Wdisk-trees}
Let $A=W_1,\ldots,W_m\looparrowright X$ be an oriented properly immersed surface supporting a Whitney tower $\cW$ in an oriented $4$--manifold $X$, with $W_i$ the connected components of $A$.

\begin{enumerate}

\item
To each $W_i$ is associated
the order zero rooted tree consisting of an edge with one vertex labeled by $i$. Recursively, the rooted tree $(I,J)$ is associated to any Whitney disk $W_{(I,J)}$ in $\cW$ pairing intersections
between $W_I$ and $W_J$ (Figure~\ref{WdiskIJandIJKint-fig}, left). 

\item
To each transverse intersection $p\in W_{(I,J)}\cap W_K$ between any $W_{(I,J)}$ and
$W_K$ in $\cW$ is associated the un-rooted tree $t_p:=\langle (I,J),K \rangle$  (Figure~\ref{WdiskIJandIJKint-fig}, right).

\end{enumerate}
\end{defn}

\begin{figure}[ht!]
        \centerline{\includegraphics[width=150mm]{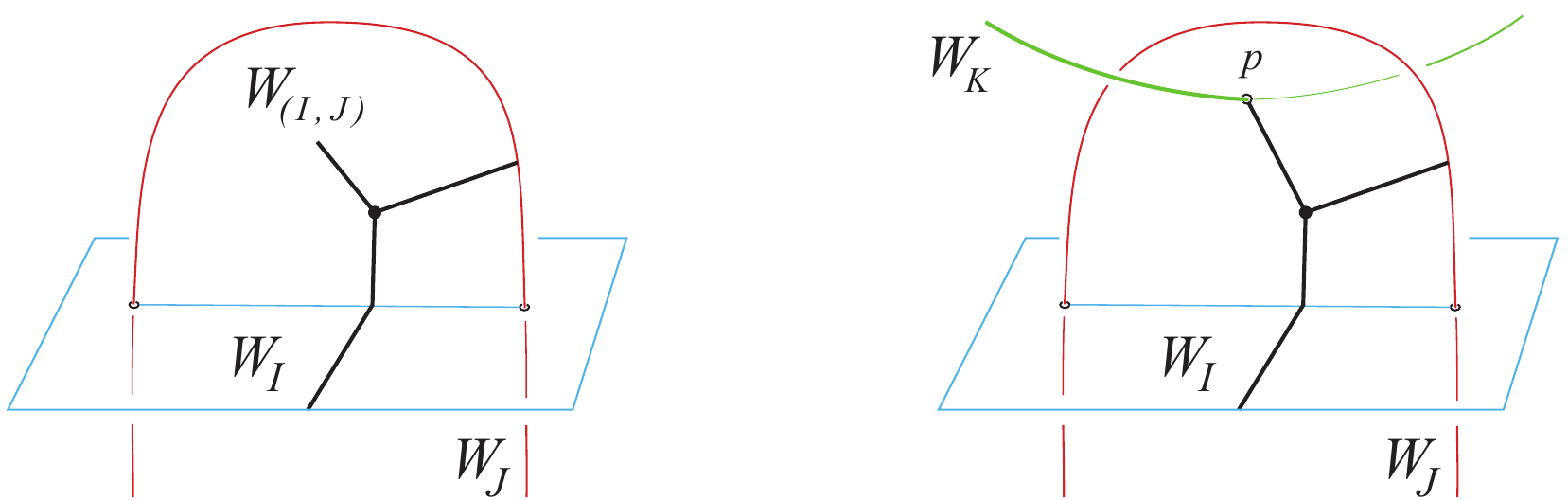}}
        \caption{}
%        \caption{On the left, (part of) the rooted tree $(I,J)$ associated to a Whitney disk $W_{(I,J)}$, with the rooted edge and adjacent trivalent vertex embedded in the interior $W_{(I,J)}$ (the rest of the edges are sheet-changing paths across Whitney disk boundaries between trivalent vertices in each Whitney disk interior). On the right, (part of) the unrooted tree $t_p=\langle (I,J),K \rangle$ associated to an intersection $p\in W_{(I,J)}\cap W_K$. Note that $p$ corresponds to where the roots of $(I,J)$ and $K$ are identified to a (non-vertex) point in $\langle (I,J),K \rangle$.
%        Also, $I$, $J$ and $K$ need not be distinct.}
        \label{WdiskIJandIJKint-fig}
\end{figure}
The left side of Figure~\ref{WdiskIJandIJKint-fig} shows (part of) the rooted tree $(I,J)$ associated to a Whitney disk $W_{(I,J)}$, with the rooted edge and adjacent trivalent vertex embedded in the interior of $W_{(I,J)}$. The rest of the edges of $(I,J)$ are sheet-changing paths across Whitney disk boundaries between the other trivalent vertices of $(I,J)$ in the interiors of each Whitney disk that has non-empty intersection with the trees $I$ and $J$ associated to $W_I$ and $W_J$. The right side of the same figure shows (part of) the unrooted tree $t_p=\langle (I,J),K \rangle$ associated to an intersection $p\in W_{(I,J)}\cap W_K$. Note that $p$ corresponds to where the roots of $(I,J)$ and $K$ are identified to a (non-vertex) point in $\langle (I,J),K \rangle$.
Also, $I$, $J$ and $K$ need not be distinct.

From these associated trees come the gradings by \emph{order} of the connected components and Whitney disks, and their transverse intersections:
\begin{defn}\label{def:w-disk-and-int-order}
For all $W_I$ and $W_J$ in $\cW$:
\begin{enumerate}
\item
The \emph{order} of $W_I$ is defined to be the order of the rooted tree $I$.
\item
The \emph{order} of a transverse intersection $p\in W_I\cap W_J$ is defined to be the order of the tree $t_p=\langle W_I,W_J\rangle$.
\end{enumerate}
\end{defn}

If $W_{J'}$ contains a vertex of the tree $J$ associated to a Whitney disk $W_J$, then we say that $W_{J'}$ \emph{supports} $W_J$.

%%%%%%%%%%%%%%%%%%%%%%%%%%%%%%%%%%%%%%%%%%%%%%%%%%%
\subsection{Whitney tower intersection forests}\label{subsec:int-forests}
The trees associated to unpaired intersections and twisted Whitney disks, together with signs and twistings, capture the essential geometric information contained in a Whitney tower:
\begin{defn}\label{def:intersection forests}
The (oriented) \emph{intersection forest} $t(\cW)$ of a Whitney tower $\cW$ is the disjoint union of signed trees associated to all unpaired intersections $p$ in $\cW$ and integer-coefficient $\iinfty$-trees associated to all twisted Whitney disks $W_J$ in $\cW$:
\[
t(\cW)=\amalg \ \epsilon_p \cdot  t_p \,\, + \amalg \ \omega(W_J)\cdot  J^\iinfty
\]
with $\epsilon_p\in\{+,-\}$ the usual sign of the transverse intersection point $p$, and $\omega(W_J)\in\Z$ the twisting of $W_J$. 
\end{defn}
More precisely, $t(\cW)$ is the \emph{multiset} of signed oriented trees and integer-coefficient $\iinfty$-trees associated to the un-paired intersections and non-trivially twisted Whitney disks in $\cW$. For cosmetic convenience, $t(\cW)$ will often be written as a formal linear combination of trees, e.g.~the ``$+$'' sign in the equation defining $t(\cW)$.

\subsection{Whitney towers of order $n$}\label{subsec:w-tower-order}

Intersection forests can be used to organize Whitney towers by \emph{order} as in the following definition, which will be used in Section~\ref{sec:proof-thm:twisted-self-Ck-equals-Milnor-Arf-equals-tower}. (This also defines the \emph{degree} of a Whitney tower as one more than the order.)
\begin{defn}\label{def:w-tower-order}\
\begin{enumerate}
\item
$\cW$ is an \emph{order $n$ (framed) Whitney tower} if every framed tree in $t(\cW)$ is of order~$\geq n$,
and every twisted tree in $t(\cW)$ is of order~$\geq n+1$. 
\item
$\cW$ is an \emph{order $n$ twisted Whitney tower} if every framed tree in $t(\cW)$ is of order~$\geq n$, and every twisted tree in $t(\cW)$ is of order~$\geq \frac{n}{2}$. 
\end{enumerate}
\end{defn}
So an order $n$ framed Whitney tower is also an order $n$ twisted Whitney tower,
and the difference between the weaker notion of order for twisted Whitney towers is that an order $n$ twisted Whitney
tower is allowed to have twisted Whitney disks of order as low as $n/2$, whereas an order $n$ framed Whitney tower is required to have only framed Whitney disks through order $n$. Motivation for allowing order $n/2$ twisted Whitney disks in an order $n$ twisted Whitney tower will appear in sections~\ref{subsec:W-tower-int-trees} and~\ref{subsec:eta-map}.

Using standard manipulations of Whitney towers (as in the proof of Lemma~\ref{lem:collapse-univalent-edges} of section~\ref{subsec:k-rep-w-towers-and-self-Ck-conc}), any order $n$ Whitney tower $\cW$ can always be modified so that $t(\cW)$ contains only framed trees of order exactly $n$,
and any order $n$ twisted Whitney tower $\cW$ can be modified so that $t(\cW)$ contains only framed trees of order exactly $n$ and twisted trees of order exactly $n/2$.

\subsection{Split Whitney towers}\label{subsec:split-w-towers}
The \emph{singularities} of a Whitney tower are the transverse intersections and self-intersections, as well as the arcs of intersections between the boundary of each Whitney disk and the interiors of the lower-order surfaces containing the boundary arcs.

A Whitney disk $W$ in a Whitney tower $\cW$ is \emph{clean} if the interior of $W$ contains no singularities (i.e.~is embedded and disjoint from the rest of $\cW$).

\begin{defn}\label{def:split-w-tower}
A Whitney tower $\cW$ is \emph{split} if both the following hold:
\begin{enumerate}
\item\label{lem-item:Bing-Hopf-split-framed}
The set of singularities in the interior
of any Whitney disk in $\cW$ consists of either a single point, or a single boundary arc of a Whitney disk, or is empty.
\item\label{lem-item:Bing-Hopf-split-twisted}
All non-trivially twisted Whitney disks in $\cW$ are clean and have twisting $\pm 1$.
\end{enumerate}
\end{defn}

So if $p\in W_I\cap W_J$ is an unpaired intersection in a split Whitney tower, then $W_I$ and $W_J$ are framed, explaining
why the associated trees $t_p=\langle I,J\rangle$ are called ``framed'' trees. Note that any framed Whitney disks which are clean can be eliminated by Whitney moves (on themselves), and this is usually tacitly assumed during manipulations of Whitney towers. The following lemma is useful for proofs involving manipulations of Whitney towers, and will be invoked frequently in subsequent sections:

\begin{lem}\label{lem:split-w-tower}
If $A$ supports a Whitney tower $\cW$ in a $4$-manifold $X$, then $A$ is homotopic (rel $\partial$) by finger-moves to 
$A'$ supporting a split
Whitney tower $\cW'$ with $t(\cW')= t(\cW)$.$\hfill\square$
\end{lem}
A proof of this result can be found in Lemma~2.18 of \cite{CST1}; and in the
case where all Whitney disks are framed in \cite[Lem.3.5]{S1} or \cite[Lem.13]{ST2}.

%We note that for the purposes of making $t(\cW)$ embedded in $\cW$, the achieving of unit twistings in item~(\ref{lem-item:Bing-Hopf-split-twisted}) of Definition~\ref{def:split-w-tower} is not necessary, and all twisted Whitney disks can be arranged to be clean with integral twistings.   

%LEMMA HERE ON MOVING THE PUNCTURE...?

%%%%%%%%%%%%%%%%%%%%%%%%%%%%%%%%%%%%%%%%%%%%%%%%%
\section{Claspers}\label{sec:claspers}
This section fixes terminology, notation and conventions on (twisted) claspers, and summarizes the relevant relationship between clasper surgery and grope cobordism. 
 
\subsection{Clasper conventions}\label{subsec:clasper-conventions}

For details on Habiro's {\em clasper surgery} techniques see \cite{CT1,CT2,Hab00}. We adopt the terminology of \cite{CT1,CT2}, together with that of \cite{CST7} for twisted claspers. Although claspers are surfaces, unless otherwise specified we follow the customary identification of a clasper with its $1$-dimensional spine, which is a framed unitrivalent graph. All of our claspers will be \emph{tree claspers}, which means that collapsing each leaf to a point yields a unitrivalent tree, sometimes referred to as ``the underlying tree'' of the clasper.

\begin{defn}\label{def:claspers}
For claspers and links in $S^3$ we have the following terminology:
\begin{enumerate}

\item
A clasper $\Gamma$ is \emph{capped} if the leaves of $\Gamma$ bound disjointly embedded disks (the \emph{caps}) into $S^3\setminus \Gamma$. (So each leaf of a capped clasper is unknotted.) 

\item
A cap for a clasper on a link $L$ is called a \emph{simple cap} if it is $0$-framed and intersects $L$ in a single point.

\item
A \emph{simple clasper} $\Gamma$ is a capped tree clasper on a link such that each cap of $\Gamma$ is simple.  

\item
A \emph{twisted clasper} $\Gamma^\iinfty$ is a capped tree clasper on a link $L$ such that all caps are simple except for one 
$\omega$-framed cap, for some $\omega\neq 0$, whose interior is disjoint from $L$. 
If $\Gamma^\iinfty$ is a twisted tree clasper, this \emph{twisting} $\omega\in\Z\setminus\{0\}$ of $\Gamma^\iinfty$ will be denoted $\omega(\Gamma^\iinfty)$.

\item
The tree $t(\Gamma)$ associated to a simple clasper $\Gamma$ is the underlying tree of $\Gamma$ with each univalent vertex labeled by the index of the link component that intersects the corresponding cap of $\Gamma$.

\item
The twisted tree $t(\Gamma^\iinfty)$ associated to a twisted clasper $\Gamma^\iinfty$ is the underlying tree of $\Gamma^\iinfty$ with each univalent vertex corresponding to simple cap of $\Gamma^\iinfty$ labeled by the index of the link component that intersects the cap, and the univalent vertex corresponding to the twisted cap labeled by the twist symbol $\iinfty$. 

\end{enumerate}
\end{defn}

%Remark that twistings can be split into all unit twistings?

%%%%%%%%%%%%%%%%%%%%%%%%%%%%%%%%%%%%%%%%%%%%%%%%%%%%%%%%%%%%%%%%%%%%%%%%%%%%%%%%%

%%%%%%%%%%%%%%%%%%%%%%%%%%%%%%%%%%%%%%

%%%%%%%%%%%%%%%%%%%%%%%%%%%%%%%%%%%%%%%%

\subsection{Clasper intersection forests}\label{subsec:clasper-forest}
Let $\Gamma$ be a simple tree clasper on an oriented link $L$, with a choice of orientation for $\Gamma$ as a surface. This orientation induces a cyclic orientation on the trivalent vertices of the tree $t(\Gamma)$, and also defines a normal direction to the surface $\Gamma$. Define the sign 
$\epsilon_{\Gamma}\in\{+,-\}$ of $\Gamma$ to be the product of the signs of all the leaves of $\Gamma$, where the sign of a leaf is $\pm 1$ according to whether or not the link component passes through the leaf in the same or opposite direction as the normal direction to $\Gamma$.
Then $\epsilon_{\Gamma}$ does not depend on the choice of orientation on $\Gamma$ (see \cite[Lem.25]{CST}).

\begin{defn}\label{def:clasper-forests}
The (oriented) \emph{intersection forest} $t(\cC)$ of a collection $\cC=\Gamma_i\amalg\Gamma^\iinfty_j$ of simple claspers $\Gamma_i$ and twisted claspers $\Gamma^\iinfty_j$ is the disjoint union of signed trees and integer-coefficient $\iinfty$-trees:
\[
t(\cC)=\amalg \ \epsilon_{\Gamma_i}\cdot t(\Gamma_i) \,\, + \amalg \ \omega(\Gamma_j^\iinfty)\cdot  t(\Gamma_j^\iinfty).
\] 
\end{defn}
More precisely, $t(\cC)$ is the \emph{multiset} of signed framed trees and integer-coefficient $\iinfty$-trees associated to the simple and twisted claspers in $\cC$ (compare with Definition~\ref{def:intersection forests}). For cosmetic convenience, $t(\cC)$ will often be written as a formal linear combination of trees, e.g.~the ``$+$'' sign in the equation defining $t(\cC)$. 

\subsection{Twisted $C_k$-concordance and twisted self $C_k$-concordance}\label{subsec:twisted-(self)-Ck-conc-def}

Links $L$ and $L'$ are \emph{twisted $C_k$-concordant} if $L$ and $L'$ are related by a sequence of concordances and/or clasper surgeries along collections
$\cC_i$ of claspers such that $t(\cC_i)$ contains only framed trees of degree $\geq k$ and/or twisted trees of degree $\geq k/2$ for each~$i$.

Links $L$ and $L'$ are \emph{twisted self $C_k$-concordant} if $L$ and $L'$ are twisted $C_k$-concordant via claspers whose intersection forests contain only \emph{mono-labeled trees}. Here a mono-labeled framed tree has all univalent vertices labeled by the same element of the index set, and a mono-labeled twisted tree has all univalent vertices labeled by the same element of the index set except for the univalent vertex which is labeled by the twist symbol $\iinfty$.

\subsection{Claspers and gropes}\label{subsec:claspers-and-gropes}
We refer the reader to e.g.~\cite{CT1,CST} for details on gropes and their associated trees, as well as the notions of capped grope cobordism and grope concordance of links. Briefly, a \emph{capped grope cobordism} between links is a grope embedded in $3$-space whose bottom stage is a collection of annuli each cobounded by the corresponding link components, such that
the tips of the grope bound embedded caps whose interiors each only intersect a single component of the bottom stage annuli. 
And a \emph{grope concordance} between links is a grope embedded in $S^3\times I$ cobounded by the links in each end of $S^3\times I$.
By ``pushing down'' intersections and ``cap splitting'', the tips of a grope concordance can always be arranged to bound caps whose interiors each only intersect a single component of the bottom stage annuli, so both capped grope cobordisms and grope concordances have associated trees whose labels come from the bottom stage components cobounded by the link components. 
One gets \emph{twisted} notions of capped grope cobordism and grope concordance by allowing each grope branch to have at most one twisted clean cap.

The following special case of the main result of \cite{CT1} adapted to links (as in \cite[Thm.23]{CST}) will be used during the proof of Theorem~\ref{thm:w-tower-forest-implies-clasper-forest}:

\begin{thm}\label{thm:clasper-surgery-equals-grope-cobordism}
A link $L$ is the result of clasper surgery along a collection $\cC$ of claspers on the unlink $U$ if and only if 
$L$ and $U$ are capped grope cobordant by a capped grope $G$ with $t(G)=t(\cC)$.
\end{thm}
Here $t(G)$ is the collection of signed trees associated to $G$, see \cite[Sec.3.4]{CST}.

Links $L$ and $L'$ are \emph{clasper concordant} if $L$ and $L'$ are related by a sequence of clasper surgeries and/or concordances. 
The following theorem follows from pushing the capped grope cobordisms from Theorem~\ref{thm:clasper-surgery-equals-grope-cobordism} into $4$-space as in \cite[Sec.3.5]{CST}:
\begin{thm}\label{thm:clasper-concordance-equals-grope-concordance}
If $\cC_1,\cC_2,\ldots,\cC_r$ are the collections of claspers in a clasper concordance from $L$ to $L'$, then there exists a grope concordance $G$ from $L$ to $L'$ such that $t(G)=\amalg^r_{i=1} t(\cC_i)$. 
\end{thm}
Theorem~\ref{thm:clasper-concordance-equals-grope-concordance} will be used in the proof of Theorem~\ref{thm:clasper-forest-implies-w-tower-forest}.

%SOMEWHERE: The notion of intersection forest in this and recent papers is referred to in \cite{CST} and ... as a ``geometric intersection tree''...

%%%%%%%%%%%%%%%%%%%%%%%%%%%%%%%%%%%%%%%%%%%%%%%%
\section{Iterated Bing-doubles}\label{sec:bing-doubles}
This section describes relationships among Bing-doubling, Whitney towers and clasper surgery that will be used in Section~\ref{sec:Ck-conc-and-w-towers}.

The following generalization of the usual indexing for links and Whitney towers will be useful for describing links and Whitney towers that temporarily appear during intermediate steps in constructions. 
\subsection{Colored links and Whitney towers}\label{subsec:colored-links-and-towers}

A \emph{colored link} is a link $L$ together with a ``coloring'' map from the components of $L$ to the set $\{1,2,\ldots,m\}$
which is not necessarily a bijection (cf.~\cite{FT1}). A coloring is indicated by subscripts, eg.~$L_i$ denotes a component of $L$ in the preimage of $i$. 

Similarly, a \emph{colored Whitney tower} $\cW$ is a Whitney tower on a properly immersed surface $A$ equipped with a map from the connected components of $A$ to the set $\{1,2,\ldots,m\}$ which is not necessarily a bijection.

\subsection{Bing-doubling along trees}\label{subsec:Bing-along-trees}

See \cite[sec.3.2]{CST1} for details %(including orientations and signs) 
on the following association of trees to iterated Bing-doubles, which is based on Tim Cochran's ``Bing-doubling along a tree'' construction \cite{C}.

Let $L$ be a colored iterated untwisted Bing-double of the $\pm$-Hopf link.
A signed framed tree $t(L)$ corresponds to such an $L$ as follows:
If $L$ is just the $\pm$-Hopf link $L_i\cup L_j$ (Bing-doubling zero times), then $t(L)=\pm\langle i,j\rangle$.
Inductively, if $L'$ is gotten from $L$ by Bing-doubling a component $L_r$ of $L$, then 
$t(L')$ is gotten from $t(L)$ by attaching a pair of new univalent edges to the univalent vertex of $t(L)$ that corresponds to $L_r$, and then labeling each of the two new univalent vertices by the elements of $\{1,2,\ldots,m\}$ which color the corresponding components of $L'$. 
%SIGNS, ORIENTATIONS?

This correspondence between $L$ and $t(L)$ is described by saying ``$L$ is gotten by Bing-doubling a Hopf link along $t(L)$''.

%REMARK: THIS ASSOCIATION ACTUALLY KEEPS TRACK OF THE PUNCTURED EDGE OF THE FRAMED TREE.
%WE WON'T REALLY NEED THIS EXTRA DATA...CAN MOVE PUNCTURE TO ANY EDGE IN WHITNEY TOWER, CORRESPONDS TO MOVING EDGE IN BING-DOUBLING CONSTRUCTION...

Now let $L$ be a colored iterated untwisted Bing-double of a twisted Bing-double of the unknot, i.e.~$L$ is constructed by starting with a $\omega$-twisted Bing-double of the unknot $L_i\cup L_j$, for some integer $\omega\neq 0$, and then performing finitely many untwisted Bing-doublings of $L_i$ and/or $L_j$.
A signed twisted tree $T^\iinfty(L)$ corresponds to such an $L$ as follows: If $L=L_i\cup L_j$ is the $\omega$-twisted Bing-double of the unknot, then $T^\iinfty(L)=\omega\cdot(i,j)^\iinfty$. 
The inductive step is the same as for framed trees: 
If $L'$ is gotten from $L$ by Bing-doubling a component $L_r$ of $L$, then 
$t(L')$ is gotten from $T^\iinfty(L)$ by attaching a pair of new univalent edges to the univalent vertex of $T^\iinfty(L)$ that corresponds to $L_r$, and then labeling each of the two new univalent vertices by the elements of $\{1,2,\ldots,m\}$ which color the corresponding components of $L'$.

This correspondence between $L$ and $T^\iinfty(L)$ is described by saying ``$L$ is gotten by Bing-doubling a $\omega$-twisted unknot along $T^\iinfty(L)$''.

%ORIENTATIONS,SIGNS,...

 %%%%%%%%%%%%%%%%%%%%%%%%%%%%%%%%%%%%%%%%%%%%%%%%
%%%%%%%%%%%%%%%%%%%%%%%%%%%

\subsection{Iterated Bing-doubles in Whitney towers}\label{subsec:bing-in-w-towers}
The two following lemmas show how Whitney towers in $4$--manifolds can be decomposed into local standard Whitney towers in $4$--balls bounded by colored links which are iterated Bing-doubles of Hopf links or iterated Bing-doubles of twisted Bing-doubles of the unlink. These lemmas will be used in the proof of Theorem~\ref{thm:w-tower-forest-implies-clasper-forest} in Section~\ref{sec:Ck-conc-and-w-towers}.

\begin{lem}\label{lem:framed-tree-bing-link}
Let $\cW\subset X$ be a split Whitney tower in a $4$--manifold $X$, and let $\epsilon_p\cdot t_p\in t(\cW)$ be an order $n$ tree corresponding to $p=W_I\cap W_J$, with $t_p$ embedded in $\cW$. Then there exists a regular $4$--ball neighborhood $N\subset X$ of $t_p$ such that:
\begin{enumerate}
\item
$L^N:=\partial N\cap \cW$ is an $(n+2)$-component colored iterated Bing double of a Hopf link along $\epsilon_p\cdot t_p$, and
\item
$N\cap \cW$ is a colored split Whitney tower in $N$ on embedded disks bounded by $L^N$ with $t(N\cap \cW)=\epsilon_p\cdot t_p$.
\end{enumerate}
\end{lem}
%The puncture of $t_p$ can preserved; or moved to any edge...
%Orientations, signs?
\begin{proof}
First consider the special case $n=0$, with $t_p=\,i-\!\!\!-\,j$ a small sheet-changing arc passing through $p\in W_i\cap W_j$: For $N$ a small $4$--ball around $p$ containing $t_p$, $N\cap \cW=D_i(p)\cup D_j(p)$ is the union of embedded disks $D_i(p)\subset W_i$ and $D_j(p)\subset W_j$ intersecting transversely in $p$, and $L^N:=\partial N\cap \cW$ is a Hopf link  $H_p=\partial D_i(p)\cup \partial D_j(p)$ (Bing-doubled zero times) colored by $i$ and $j$. Here the sign $\epsilon_p$ of $t_p$ 
is the sign of $H_p$, via the usual orientation conventions.

Now assume $t_p$ has positive order $n>0$.
Take the embedding of $t_p$ in $\cW$ to have each $i$-labeled univalent edge terminate at the order $0$ surface $W_i$ (rather than changing sheets into $W_i$), so if a Whitney disk $W_{(i,J)}$ contains a trivalent vertex of $t$, then the adjacent $i$-labeled univalent vertex of $t_p$ is in the boundary arc of $W_{(i,J)}$ that lies in $W_i$.

Take $N$ to be
a small regular $4$--ball neighborhood of the union of the Whitney disks containing the trivalent vertices of $t_p$ (and hence containing all of $t_p$ by the previous paragraph's choice of the embedding $t\subset\cW$). By taking $N$ small enough it may be arranged that $N$ intersects the order $0$ disks $W_i$ in embedded disk-neighborhoods $\Delta_i^r\subset W_i$ around each boundary arc $\partial W_{(i,J)}\subset W_i$, and around $p$ if $p\in W_i$. The superscript $r$ for $\Delta_i^r$ ranges over the number of univalent vertices labeled by $i$ in $t_p$. 

%Note that the embedded disks $\Delta_i^r\subset W_i$ may be assumed to also be pair-wise disjoint, except for the possible cases where $\Delta_i^r\cap\Delta_i^s$ for $r\neq s$ are transverse self-intersections $W_i\cap W_i$ with $t_p$ passing through $W_{(i,i)}\subset\cW$.

The link
$L^N=\cup_{i,r}\partial\Delta_i^r=\partial N\cap\cW$ has $n+2$ components, one component for each univalent vertex of $t_p$.  
To see that $L^N$ is an iterated Bing-double of a Hopf link along $t_p$, we start with the following general observation:

\textbf{Observation:}
Let $W$ be an embedded Whitney disk pairing oppositely-signed transverse intersections $p$ and $q$ between surface sheets $U$ and $V$ in a $4$--manifold $X$. The normal disk bundle $\nu_W\cong W\times D^2$ of $W$ in $X$ is diffeomorphic to a $4$--ball $B$; and via the identification of $\nu_W$ with a tubular neighborhood of $W$ in $X$, it may be assumed that $B\cap U$ is an embedded disk-neighborhood $\Delta_U\subset U$ of the boundary arc $\partial_U W\subset U$, and $B\cap V$ is an embedded disk-neighborhood $\Delta_V\subset V$ of $\partial_V W\subset V$.     
The disk $\Delta_U$ decomposes as the union of disks $D_U(p)$ and $D_U(q)$ around $p$ and $q$, together with a band $b_U$ running along $\partial_U W$.
Similarly, the disk $\Delta_V$ decomposes as the union of disks $D_V(p)$ and $D_V(q)$ around $p$ and $q$, together with a band $b_V$ running along $\partial_V W$. 

The link $L^B=\partial\Delta_U\cup\partial\Delta_V\subset \partial B$ is a band sum of two oppositely signed Hopf links $H_p=\partial D_U(p)\cup \partial D_V(p)$ and $H_q=\partial D_U(q)\cup \partial D_V(q)$. It follows that $L^B$ is a (possibly twisted) Bing-double of a knot; and this knot must be the unknot, since an unknotting disk is exhibited by a parallel of the embedded $W$ pushed into $\partial B$ (using the product structure $B\cong W\times D^2$). The twisting of the Bing-doubling of $L^B$ is 
equal to the twisting of $W$ relative to $U$ and $V$ around $\partial_U W\cup \partial_V W$, which is
exactly $\omega (W)$.

Returning to consideration of $L^N$: 

Since $t_p=\langle W_I,W_J\rangle$ has positive order, at least one of $W_I$ and $W_J$ is a Whitney disk, say $W_I=W_{(U,V)}$ (recall that all Whitney disks containing trivalent vertices of $t$ are framed and embedded since $\cW$ is split). Now applying the above Observation to a $4$--ball thickening $B\subset N$ of $W_{(U,V)}$,
we see in $\partial B\cap N$ an untwisted Bing-double $L^B=L^B_U\cup L^B_V$ of the unknotted boundary of $W_{(U,V)}$, with $L^B_U\subset W_U$ and $L^B_V\subset W_V$. If $W_U$ (resp. $W_V$) is order zero, then we may assume
$L^B_U=\partial\Delta_u^r$ (resp. $L^B_V=\partial\Delta_v^r$) for some $r$.  Otherwise, say $W_U=W_{(U_1,U_2)}$ is of positive order. Then $B$ can be extended to a $4$--ball $B'$ which also contains a thickening of $W_U$,
with $L^B_U=\partial W_U$, and the resulting link $L^{B'}=B'\cap N$ is gotten from $L^B$ by Bing-doubling the component $L^B_U$. This argument can be iterated until $B$ is extended to $N$, exhibiting the link $L^N=\cup_{i,r}\partial\Delta_i^r$ as
an iterated Bing double of the Hopf link along the %(punctured) 
tree $t_p=\langle W_I,W_J\rangle$. 
\end{proof}

\begin{lem}\label{lem:twisted-tree-bing-link}
Let $\cW\subset X$ be a split Whitney tower in a $4$--manifold $X$, and let $J^\iinfty\in t(\cW)$ be a twisted order $n$ tree corresponding to a clean twisted Whitney disk $W_J$, with $\omega(W_J)\in\Z$. Then there exists a regular $4$--ball neighborhood $N\subset X$ of $J^\iinfty$ such that:
\begin{enumerate}
\item
$L^N=\partial N\cap \cW$ is an $(n+2)$-component colored iterated untwisted Bing-double of the twisted-Bing double of the unknot along the tree $J^\iinfty$ with twisting $\omega(W_J)$, and
\item
$N\cap \cW$ is a colored split Whitney tower in $N$ on embedded disks bounded by $L^N$ with $t(N\cap \cW)=\omega(W_J)\cdot J^\iinfty$.
\end{enumerate}
\end{lem}
\begin{proof}
Starting with the Observation in the proof of Lemma~\ref{lem:framed-tree-bing-link}, a regular $4$--ball neighborhood $B$ of the embedded $W_J=W_{(J_1,J_2)}$ intersects $W_{J_1}$ and $W_{J_2}$ in an $\omega(W_J)$-twisted Bing double of the unknot. The rest of the proof of Lemma~\ref{lem:framed-tree-bing-link} applies essentially word for word to complete the proof of this lemma.
\end{proof}

\subsection{Iterated Bing-doubles and clasper surgery}\label{subsec:bing-clasper} 
It is well-known that the effect of tree clasper surgery on an unlink corresponds to iterated Bing doubling \cite[Sec.7.1]{Hab00}. 
This follows from the observation that a Y-clasper surgery ties three strands of a link into a local Borromean Rings which is a Bing double of a Hopf link, see e.g.~\cite[Fig.34]{Hab00}. And surgery on a twisted Y-clasper with one clean $\omega$-twisted leaf ties the other two strands into
an $\omega$-twisted Bing-double of the unknot, see e.g.~\cite[Move 10]{Hab00}.

For future reference we state this as a lemma:
\begin{lem}\label{lem:clasper-realizes-bing-double}
If a link $L$ is a band-sum of iterated Bing-doubles
along a linear combination $z$ of trees (possibly including some twisted trees), then $L$ is obtained by clasper surgery along a collection $\cC$ of claspers on an unlink such that $t(\cC)= z$.$\hfill\square$
\end{lem} 
%PROOF? TWISTED CASE?
%%%%%%%%%%%%%%%%%%%%%%%%%%%%%%%%%%%%%%%%%%%%%%%%%%%%%

\section{Clasper concordance and Whitney towers}\label{sec:Ck-conc-and-w-towers}
Theorem~\ref{thm:clasper-concordance} in the introduction follows from Theorem~\ref{thm:w-tower-forest-implies-clasper-forest} and Theorem~\ref{thm:clasper-forest-implies-w-tower-forest} in this section.

%Recall that a clasper concordance is a sequence of clasper surgeries and concordances.

\subsection{From Whitney towers to clasper concordance}\label{subsec:w-tower-to-clasper-concordance}

The following theorem implies the ``only if'' direction of Theorem~\ref{thm:clasper-concordance}:
\begin{thm}\label{thm:w-tower-forest-implies-clasper-forest}
If a link $L$ bounds a Whitney tower $\cW$ with intersection forest $t(\cW)$, then the unlink is clasper concordant to $L$
via a concordance, followed by clasper surgeries on a collection $\cC$ of claspers with $t(\cC)=t(\cW)$, followed by another concordance. 
\end{thm}
 %symmetry of $C_k$-equivalence: \cite[Prop.3.23]{Hab00}.

%%%%%%%%%%%%%%
\begin{rem}
Theorem~\ref{thm:w-tower-forest-implies-clasper-forest} plays a fundamental role in the proof of the main result of  \cite{CST7} which describes Cochran's $\beta$-invariants in terms of Whitney disk twistings.
\end{rem}
%%%%%%%%%%%%%%%%%%%%

\begin{figure}[ht!]
        \centerline{\includegraphics[scale=.325]{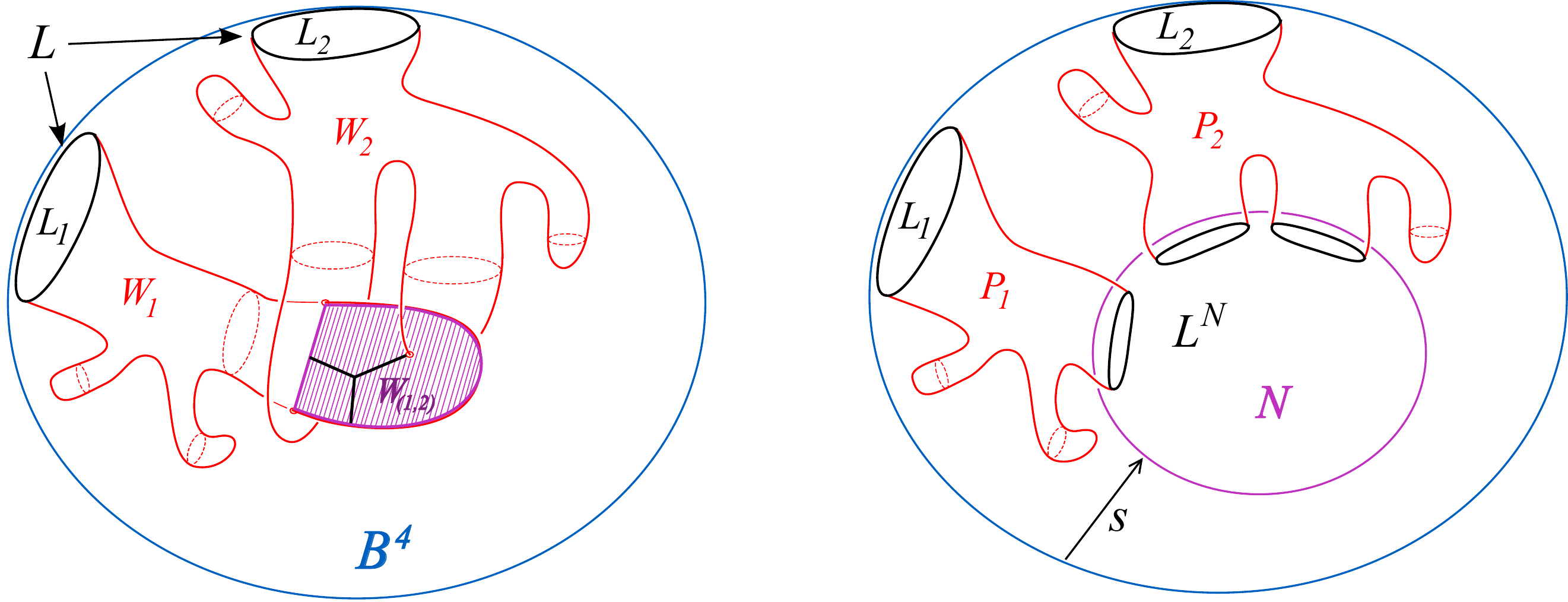}}
        \caption{Left: $L=L_1\cup L_2$ bounds $\cW$ with $t(\cW)=\pm\cdot t_p=\pm\langle (1,2),2\rangle$. Right: The link $L^N=\cW\cap \partial N$, where $N$ is a $4$--ball neighborhood of $t_p$ containing all singularities of $\cW$, and the planar surface cobordism $P=P_1\cup P_2\subset B^4\setminus \textrm{int}(N) \cong S^3\times[0,1]$ from $L$ to $L^N$. (The $s$-arrow indicates the interval factor of $S^3\times[0,1]$.)}
        \label{fig:w-tower-to-clasper-1AB}
\end{figure}

\begin{proof} 

We first consider the cases where $t(\cW)$ consists of a single framed or twisted tree.
The general case will follow easily from these cases.

\textbf{Case 1:}
Let $L$ be an $m$-component link bounding a Whitney tower $\cW$ which has just a single unpaired intersection 
point $p$, which is of order $n$, and all Whitney disks in $\cW$ are framed. So the intersection forest $t(\mathcal W)$ consists of a single signed tree $\epsilon_p\cdot t_p$ of order $n$.

As per Lemma~\ref{lem:split-w-tower}, we can assume that $\cW$ is split, so the Whitney disks of $\cW$ are each embedded, and the singularities in the interior of each Whitney disk consist of either a single boundary arc of a high-order Whitney disk or the unpaired intersection $p$.

By Lemma~\ref{lem:framed-tree-bing-link}
there exists
a regular $4$--ball neighborhood $N$ of the Whitney disks containing $t_p\subset\cW$ such that the link
$L^N=\partial N\cap\cW$ is an $(n+2)$-component colored iterated Bing-double of the Hopf link, bounding embedded disks which support the Whitney tower $N\cap\cW$ in $N$, with $t(N\cap\cW)=t_p$ (see schematic picture in Figure~\ref{fig:w-tower-to-clasper-1AB}).

\begin{figure}[ht!]
        \centerline{\includegraphics[scale=.35]{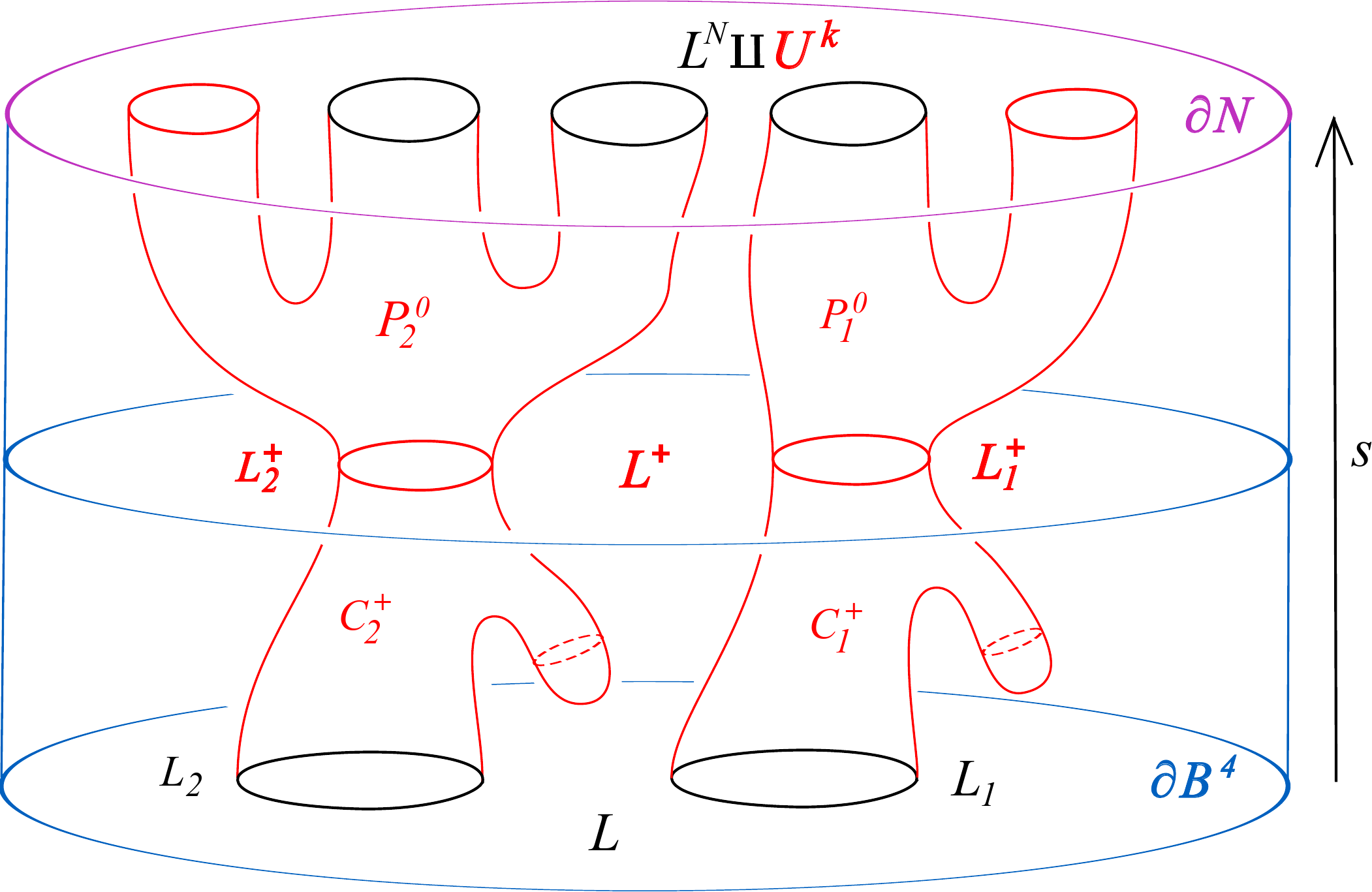}}
        \caption{The cobordism $C^+\cup P^0\subset S^3\times[0,1]$ from $L$ to $L^N\amalg U^k$.}
        \label{fig:w-tower-to-clasper-2AB}
\end{figure}
\begin{figure}[ht!]
        \centerline{\includegraphics[scale=.3]{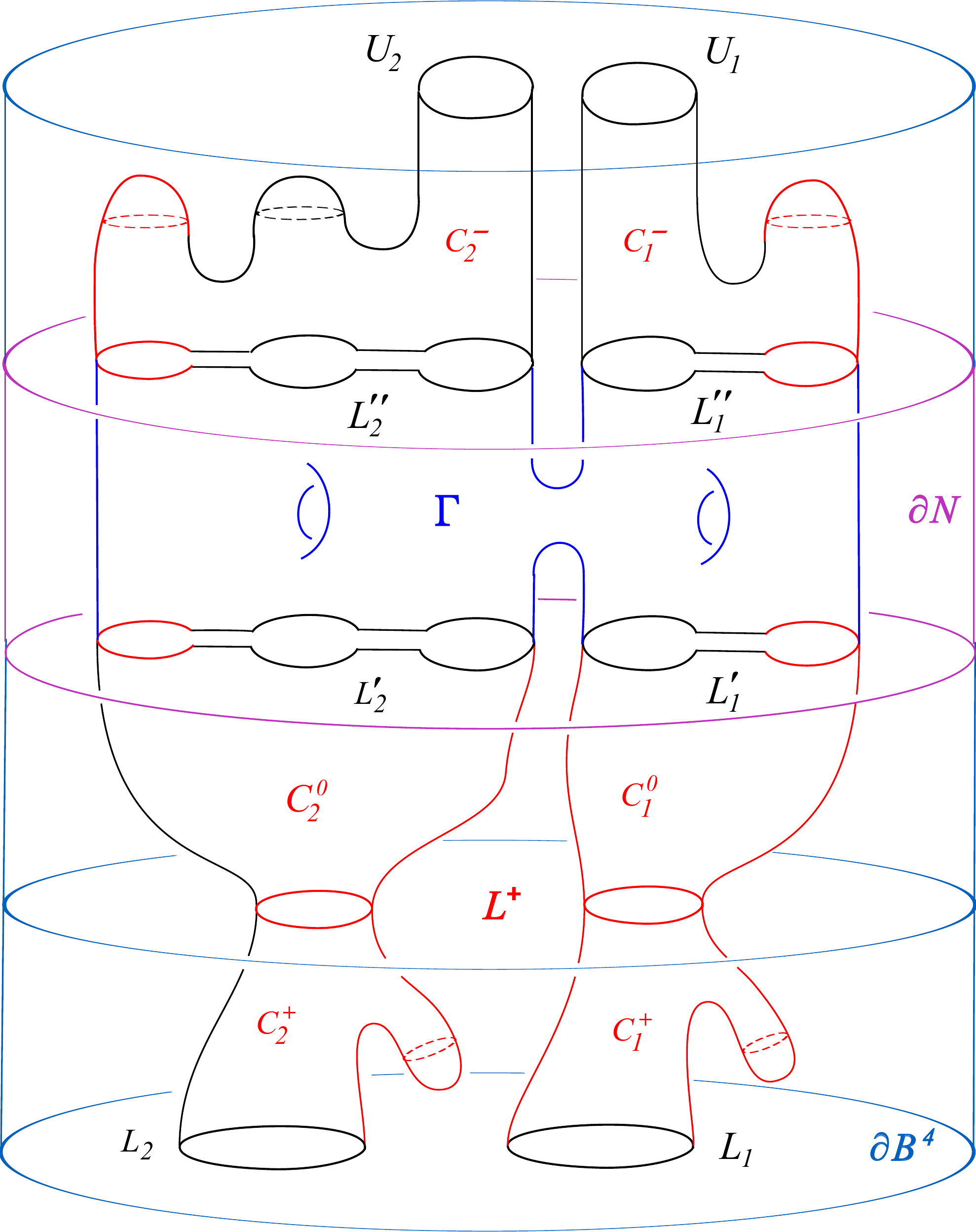}}
        \caption{The bottom section is the same as in Figure~\ref{fig:w-tower-to-clasper-2AB}. The second from bottom section shows the concordance $C^0$ from $L^+$ to $L'$ formed by pushing the punctured maxima and saddles of $P^0$ (Figure~\ref{fig:w-tower-to-clasper-2AB}) into $\partial N$ . The third section from bottom (purple) depicts clasper-surgery on $\Gamma$ as (the bottom stage of) a $3$-dimensional grope cobordism between $L'$ and $L''$ entirely contained in the $3$-dimensional slice $\partial N$. The top section shows the `band-cutting' concordance $C^-$ from $L''$ (an internal band sum of an $(n+2+k)$-component unlink) to an $m$-component unlink $U^m$.}
        \label{fig:w-tower-to-clasper-3A}
\end{figure}

Now $B^4\setminus \textrm{int}(N) \cong S^3\times[0,1]$, and $\mathcal W\cap (B^4\setminus \textrm{int}(N))$ is an $m$-component planar surface $P$ which is a cobordism from $L\subset S^3\times \{0\}$ to the iterated Bing-double $L^N\subset S^3\times \{1\}$. We can assume all minima come before all maxima. Alter $P$ by puncturing the local maxima and stretching them upward to meet $S^3\times\{1\}=\partial N$. Then we have a cobordism from $L$ to a split union $L^N\amalg U^k$, where $U^k$ is a $k$-component unlink (with $k$ the number of maxima of $P$), and this cobordism has no local maxima. We can factor this cobordism as a concordance $C^+$ from $L$ to some $m$-component link $L^+$, followed by a cobordism $P^0$ from $L^+$ to $L^N\amalg U^k$ which only has saddles (Figure~\ref{fig:w-tower-to-clasper-2AB}). Pushing the saddles of $P^0$ up into $\partial N$ yields a concordance $C^0$ from $L^+$ to an $m$-component link $L'=\#_b(L^N\amalg U^k)\subset \partial N$, which is an internal band sum of $L^N\amalg U^k$ along some collection $b$ of bands preserving components of $P^0$ (bottom two sections of Figure~\ref{fig:w-tower-to-clasper-3A}).

%\begin{figure}[ht!]
%        \centerline{\includegraphics[scale=.3]{Figures/w-tower-to-clasper-3A.pdf}}
%        \caption{The bottom section is the same as in Figure~\ref{fig:w-tower-to-clasper-2AB}. The second from bottom section shows the concordance $C^0$ from $L^+$ to $L'$ formed by pushing the punctured maxima and saddles of $P^0$ (Figure~\ref{fig:w-tower-to-clasper-2AB}) into $\partial N$ . The third section from bottom (purple) depicts clasper-surgery on $\Gamma$ as (the bottom stage of) a $3$-dimensional grope cobordism between $L'$ and $L''$ entirely contained in the $3$-dimensional slice $\partial N$. The top section shows the `band-cutting' concordance $C^-$ from $L''$ (an internal band sum of an $(n+2+k)$-component unlink) to an $m$-component unlink $U^m$.}
%        \label{fig:w-tower-to-clasper-3A}
%\end{figure}

Let $\Gamma\subset \partial N$ be a capped tree clasper on an $(n+2+k)$-component unlink $U^{n+2+k}$ such that $\Gamma(U^{n+2})=L^N$ and $\Gamma(U^{n+2+k})=L^N\amalg U^k$, with $\epsilon_\Gamma\cdot t(\Gamma)=\epsilon_p\cdot t_p$. Such a $\Gamma$ exists because $L^N$ is the $(n+2)$-component iterated Bing-double of an $\epsilon_p$-Hopf link along $t_p$ (Lemma~\ref{lem:clasper-realizes-bing-double}). 

Since the bands $b$ and $\Gamma$ can be assumed to be disjoint, $\Gamma$ is also a clasper on the 
link $L''=\sharp_bU^{n+2+k}$, with the same associated tree $t(\Gamma)=\epsilon\cdot t_p=t(\cW)$ as $\cW$, and the result %$\Gamma(L'')$
of clasper surgery by $\Gamma$ on $L''$ is equal to $L'$.

%We have
%$$\Gamma(\underset{L''}{\underbrace{\sharp_b(U^{n+2+k}}}))=\sharp_b(L^N\amalg U^k)=L'.$$ 

Now $L''$ is concordant to an $m$-component unlink $U^m$ via a concordance $C^-$ with $n+2+k-m$ minima and $n+2+k-m$ saddles which ``cuts the bands" (Figure~\ref{fig:w-tower-to-clasper-3A}). 

This completes the case where $t(\cW)=t_p$.

\textbf{Case 2:}
Now assume that $\cW$ has no unpaired intersections, and just a single twisted Whitney disk $W_J$ with twisting $\omega(W_J)\in\Z$.
This case only differs from Case 1 in that we apply Lemma~\ref{lem:twisted-tree-bing-link} instead of Lemma~\ref{lem:framed-tree-bing-link}.

\textbf{General case:}
For $t(\cW)$ consisting of multiple trees, by the above arguments there exist disjoint $4$--ball neighborhoods $N_r$ of the trees in $t(\cW)$ such that $\partial N_r\cap\cW$ are iterated Bing-doubles of Hopf links and/or of twisted Bing-doubles of the unknot. Connecting the $N_r$ by thickenings of arcs that miss $\mathcal W$ yields a single $4$-ball $N$, and the constructions proceed as before, except that $\Gamma$ will now be a collection of disjoint claspers $\Gamma_r$. 
\end{proof}

%%%%%%%%%%%%%%%%%%%%%%%%%%%%%%%%%%%%%%%%%%%%%%%%%%%%%%

\subsection{From clasper concordance to Whitney towers}\label{subsec:clasper-to-w-tower-concordance}

The following theorem implies the ``if'' direction of Theorem~\ref{thm:clasper-concordance} from the introduction. It follows directly from \cite{CT1,CT2,CST}, but we state it here for completeness, and outline the proof:
\begin{thm}\label{thm:clasper-forest-implies-w-tower-forest}
If a link $L$ is (twisted) clasper concordant to the unlink, then $L$ bounds a (twisted) Whitney tower $\cW$ with intersection forest $t(\cW)=\amalg_i\, t(\cC_i)$, where $\cC_i$ are the collections of claspers in the clasper concordance. 
\end{thm}
\begin{proof}
The (twisted) capped grope concordance between $L$ and the unlink from Theorem~\ref{thm:clasper-concordance-equals-grope-concordance} can be surgered to a (twisted) Whitney tower by \cite[Thm.6]{S1} and capped off.
(The correspondence between twisted caps and twisted Whitney disks is illustrated in e.g.~\cite[Fig.10]{CST2}.)
\end{proof}

%%%%%%%%%%%%%%%%%%%%%%%%%%%%%%%%%%%%%%%%%%%%
%if $\Gamma^\iinfty$ is a twisted clasper on a link $L$ whose $\iinfty$-tree is of degree $(k+1)/2$ then the result of surgery on $\Gamma^\iinfty$ is $C_k$-concordant to $L$.
%This follows from the proof of \cite[Lem.33]{CST5}, using \cite[Rem.2.4]{L1} to strengthen \cite[Lem.27]{CST5}.
%NOTE: The proof of \cite[Lem.33]{CST4} assumes unit twistings.
%%%%%%%%%%%%%%%%%%%%%%%%%%%%%%%%%%%%%%%%%%%%%%%%%%%%

\section{Proof of Theorem~\ref{thm:twisted-self-Ck-equals-Milnor-Arf-equals-tower}}\label{sec:proof-thm:twisted-self-Ck-equals-Milnor-Arf-equals-tower}
%%%%%%%%%%%%%%%%%%%%%%%%%%%%%%%%%%%%%%%%%%%%%%%%%%%%%%%%%
Theorem~\ref{thm:twisted-self-Ck-equals-Milnor-Arf-equals-tower} states that for an $m$-component link $L$ in the $3$--sphere the following three properties are equivalent for each positive integer $k$:
 $L$ is twisted self $C_k$-concordant to the unlink;
 $L$ bounds a $k$-repeating twisted Whitney tower of order $mk-1$; and,
for all $1\leq r\leq k$, $L$ has vanishing $r$-repeating Milnor invariants and $r$-repeating higher-order Arf invariants. 

The proof given in this section uses $k$-repeating versions of the classification of the order $n$ twisted Whitney tower filtration on links  
given in \cite{CST1,CST2,CST3,CST4}. 
Results from these papers will be summarized as needed during the proofs, with details given for the new aspects that arise in the $k$-repeating setting. Rather than recreating here the many arguments and constructions from \cite{CST1,CST2,CST3,CST4}, the necessary definitions will be given along with descriptions of how to adapt proofs to the current $k$-repeating setting. Inevitably, a complete understanding of Theorem~\ref{thm:twisted-self-Ck-equals-Milnor-Arf-equals-tower} will depend heavily on having already understood \cite{CST1,CST2,CST3,CST4}.

Here is an outline of this section:

\begin{itemize}

\item
Sections~\ref{subsec:multiplicities} fixes notions of multiplicity, section~\ref{subsec:k-rep-twisted-w-towers} defines $k$-repeating (twisted) Whitney towers and shows that a $k$-repeating twisted Whitney tower of order $mk-1$ can be modified to have only mono-labeled trees. This latter result is used in section-\ref{subsec:k-rep-w-towers-and-self-Ck-conc} to show that bounding a $k$-repeating Whitney tower of order $mk-1$ is equivalent to being self $C_k$-concordant to the unlink.

\item
Sections~\ref{subsec:intro-Milnor-review}--\ref{subsec:intro-higher-order-arf}
describe the classification of the twisted Whitney tower filtration on links by order in terms of Milnor invariants and higher-order Arf invariants, as summarized in Corollary~\ref{cor:intro-mu-arf-classify-twisted}.

\item
In sections~\ref{subsec:k-free-lie}--\ref{subsec:k-repeating-twisted-classification} 
this classification is adapted to the twisted $k$-repeating setting, including definitions of the first non-vanishing $k$-repeating total Milnor invariants, the $k$-repeating twisted Whitney tower obstruction theory, and the $k$-repeating higher-order Arf invariants. The proof of Theorem~\ref{thm:twisted-self-Ck-equals-Milnor-Arf-equals-tower} is completed by Corollary~\ref{cor:k-repeating-mu-arf-classify-k-repeating-twisted}.

\end{itemize}

Throughout this section the parameter $k$ denotes a positive integer.
%%%%%%%%%%%%%%%%%%%%%%%%%%%%%%%%%%%%%%%%%%%%%%%%%%%%%%OLD:

%%%%%%%%%%%%%%%%%%%%%%%%%%%%%%%%%%%%%%%%%%%%%%%%
\subsection{Multiplicities}\label{subsec:multiplicities}
Recall from Section~\ref{sec:trees} that our trees are finite, unitrivalent and vertex-oriented.  Univalent vertices come equipped with labels from the index set $\{1,2,\ldots,m\}$, except that a rooted tree has one designated univalent vertex (the root) which is indicated by not having a label, and a $\iinfty$-tree has one univalent vertex which is labeled by the twist symbol $\iinfty$.

For a tree $T$ (rooted or un-rooted) such that $T$ is not a $\iinfty$-tree, the \emph{$i$-multiplicity} $r_i(T)$ is defined to be the number of univalent vertices labeled by the index $i$, and the \emph{multiplicity} $r(T)$ is the maximum of $r_i(T)$ over the index set.

For a $\iinfty$-tree $T^\iinfty$ with underlying rooted tree $T$, the \emph{$i$-multiplicity} and \emph{multiplicity} are defined by $r_i(T^\iinfty)=r_i(\langle T,T\rangle)=2\cdot r_i(T)$ 
 and $r(T^\iinfty)=r( \langle T,T\rangle)=2\cdot r(T)$, respectively. 
 
Via the usual identification of rooted trees with non-associative bracketings, these notions of \emph{$i$-multipicity} and \emph{multiplicity} apply as well to Lie brackets and iterated commutators of generators in a free group.

Similarly, \emph{$i$-multiplicities} and \emph{multiplicities} are defined for intersection points $p$ and Whitney disks $W_J$ in a Whitney tower as the multiplicities of the corresponding trees $t_p$ and $J$, or $J^\iinfty$ if $W_J$ is twisted. 

\subsection{$k$-repeating twisted Whitney towers of order $n$}\label{subsec:k-rep-twisted-w-towers}
The following definition of twisted $k$-repeating Whitney towers of order $n$ is a relaxation of the definition of twisted Whitney towers of order $n$ given
in section~\ref{subsec:w-tower-order}:

%\begin{defn}\label{def:k-rep-w-tower}
%
%A Whitney tower $\cW$ is a \emph{$k$-repeating framed Whitney tower of order $n$} if every framed tree in $t(\cW)$ has order $\geq n$ or has multiplicity $>k$,
%and every $\iinfty$-tree in $t(\cW)$ has order $>n$.
%THINK ABOUT WHICH TWISTED W-DISKS CAN BE ALLOWED...
%
%\end{defn}
%So in a $k$-repeating framed Whitney tower of order $n$ any intersections of multiplicity $>k$ are not required to be paired by Whitney disks, even if they are of order $<n$.
%

\begin{defn}\label{def:k-rep-twisted-w-tower}
A Whitney tower $\cW$ is a \emph{$k$-repeating twisted Whitney tower of order $n$} if every framed tree in $t(\cW)$ has order $\geq n$ or has multiplicity $>k$,
and every $\iinfty$-tree in $t(\cW)$ has order $\geq \frac{n}{2}$ or has multiplicity $>k$.
\end{defn}
So in a $k$-repeating twisted Whitney tower of order $n$ any intersections of multiplicity $>k$ are not required to be paired by Whitney disks, even if they are of order $<n$, and any Whitney disks of multiplicity $>k$ are allowed to be twisted, even if they are of order $< \frac{n}{2}$.

The following result will be used to prove Proposition~\ref{prop:order-mk-1-W-tower-equals-self-Ck-concordance} in the subsequent section showing that a link bounds a $k$-repeating twisted Whitney tower of order $mk-1$ if and only if it is self $C_k$-concordant to the unlink:
\begin{lem}\label{lem:reduce-to-mono-labeled} 
If $A=W_1,\ldots,W_m$ supports a $k$-repeating twisted Whitney tower $\cW$ of order $mk-1$, then $A$ is regularly homotopic rel boundary to $A'$ supporting a $k$-repeating twisted Whitney tower $\cW'$ of order $mk-1$ such that all trees in $t(\cW')$ are mono-labeled. %Moreover, each framed tree in $t(\cW')$ has order exactly $mk-1$, and... 
\end{lem}
Note that since framed trees of order $\geq mk-1$ have $\geq mk+1$ univalent vertices, each tree in $t(\cW)$ has $j$-multiplicity $\geq k$ for some $j$.
The idea of the proof of Lemma~\ref{lem:reduce-to-mono-labeled} is to modify $\cW$ in a way that collapses away the 
other $i$-labeled edges for $i\neq j$ of each tree in $t(\cW)$ to yield the desired $\cW'$.

To state this precisely in the following lemma we introduce the
operation of ``collapsing an $i$-labeled edge'' in a tree $T$ of positive order, which means the following: Given an edge $e$ of $T$ that has a univalent vertex labeled by $i$, retract $e$ onto its trivalent vertex, which is then considered to be an interior point of the single edge formed by fusing the two edges of $T$ that were adjacent to $e$. In the notation of Definition~\ref{def:Tree-ops}, for framed trees this collapsing operation can be described as $\langle (J_1,J_2),i \rangle\mapsto\langle J_1,J_2\rangle$; and for rooted trees or $\iinfty$-trees a subtree 
of the form $(J,i)$ in the original tree goes to the subtree $J$ in the collapsed tree.

\begin{lem}\label{lem:collapse-univalent-edges}
Let $A$ support a split Whitney tower $\cW$.
\begin{enumerate}
\item\label{item:framed-collapse}
Let $t_p$ be a tree of positive order corresponding to an unpaired intersection $p\in\cW$, and let $i$ be an index such that $r_i(t_p)\geq 1$. Then $\cW$ can be modified near $t_p\subset\cW$ to yield $\cW'$ on $A'$ regularly homotopic rel boundary to $A$ such
that $t(\cW')$ contains the same trees as $t(\cW)$ except that $\epsilon_p\cdot t_p$ is replaced by a pair of oppositely signed trees $+t_q$ and $-t_q$, where $t_q$ is the result of collapsing an $i$-labeled edge of $t_p$. So in particular, $r_i( t_q)=r_i(t_p)-1$.

\item
Let $J^\iinfty$ be a twisted tree corresponding to a twisted Whitney disk $W_J\subset\cW$, and let $i$ be an index such that $r_i(J)\geq 1$. Then $\cW$ can be modified near $J^\iinfty\subset\cW$ to yield $\cW'$ on $A'$ regularly homotopic rel boundary to $A$ such
that $t(\cW')$ contains the same trees as $t(\cW)$ except that $\omega(W_J)\cdot J^\iinfty$ is replaced by either
\begin{enumerate}

\item\label{item:twisted-collapse-next-to-rooted-edge}
the signed tree $\omega(W_J)\cdot\langle I,I \rangle$, where $I$ is the result of collapsing an $i$-labeled edge of $J$ which is adjacent to the rooted edge of $J$; or 

\item\label{item:twisted-collapse-away-from-rooted-edge}
two copies of the signed $\iinfty$-tree $\omega(W_J)\cdot I^\iinfty$ plus the signed tree $\omega(W_J)\cdot\langle I,I \rangle$, where $I$ is the result of collapsing an $i$-labeled edge of $J$ which is \emph{not} adjacent to the rooted edge of $J$.

\end{enumerate}

So in particular, $r_i(I^\iinfty)=r_i(\langle I,I \rangle)=r_i(J^\iinfty)-2$.

\end{enumerate}
\end{lem}

\begin{proof}[Proof of Lemma~\ref{lem:reduce-to-mono-labeled}]
Apply Lemma~\ref{lem:collapse-univalent-edges} as needed: For each tree $T\in t(\cW)$ there exists $j$ such that $r_j(T)=r(T)\geq k+1$. So collapsing all $i$-labeled edges of $T$ for $i\neq j$ yields a tree mono-labeled by $j$, still of multiplicity $\geq k+1$. 
\end{proof}

We remark that Lemma~\ref{lem:collapse-univalent-edges} can be applied further until all trees are mono-labeled with multiplicity \emph{exactly} $k+1$.

\begin{proof}[Proof of Lemma~\ref{lem:collapse-univalent-edges}]

\textbf{Statement~(\ref{item:framed-collapse}):}
Consider first the case that $t_p=\langle J,i\rangle$ corresponds to an unpaired intersection $p=W_J\cap W_i$ between a Whitney disk $W_J$ and an order zero surface $W_i$, with $r_i(t_p)\geq 1$. Then deleting the Whitney disk $W_J$ from $\cW$ eliminates $p$ and creates two unpaired intersections $\pm q$ which used to be paired by $W_J$. The associated signed trees are $\pm t_q=\pm\langle J_1,J_2\rangle$, where $J=(J_1,J_2)$.

Now recall from section~\ref{subsec:trees-in-towers} that the tree $t_p=\langle J_1,J_2\rangle$ associated to any unpaired intersection $p\in W_{J_1}\cap W_{J_2}$ 
can be embedded $t_p\subset\cW$ with $p$ as an interior point of an edge of $t_p$. The reason that the extra data of this ``preferred'' edge containing $p$ is not taken into account in the intersection forest $t(\cW)$ is that local modifications can ``move'' the unpaired intersection onto any chosen edge of $t_p$, as described in \cite[Lem.14]{ST2} (see Figures~10 and 11 of \cite{ST2} and note that \cite{ST2} uses the dot product notation $J_1\cdot J_2$ for our inner product $\langle J_1, J_2\rangle$).  
So for any $t_p\subset \cW$ with $r_i(t_p)\geq 1$, the previous paragraph can be used to replace $t_p$ by $\pm t_q$ after applying \cite[Lem.14]{ST2} to arrange for an $i$-labeled edge of $t_p$ to contain the unpaired intersection.

\textbf{Statement~(\ref{item:twisted-collapse-next-to-rooted-edge}):}
Let $\pm(I,i)^\iinfty\in t(\cW)$ be associated to a $\pm 1$-twisted Whitney disk $W_{(I,i)}$.
Performing a boundary-twist of $W_{(I,i)}$ into $W_I$ converts $W_{(I,i)}$ into a framed Whitney disk that has a single unpaired intersection with $W_I$ (see \cite[Fig.18]{CST1} or \cite[Chap.1.3]{FQ}). This exchanges $\pm(I,i)^\iinfty$ for $\pm \langle (I,i),I\rangle$ in $t(\cW)$, which can then be exchanged for $\pm\langle I,I\rangle$ by the above proof of Statement~(\ref{item:framed-collapse}). (Alternatively, one could just do the $W_{(I,i)}$-Whitney move on $W_I$ using the original twisted $W_{(I,i)}$. This Whitney move (which is a regular homotopy) would create a single self-intersection of $W_I$ (from the unit twisting) with associated tree $\pm\langle I,I\rangle$.)

\textbf{Statement~(\ref{item:twisted-collapse-away-from-rooted-edge}):}
Let $\pm J^\iinfty\in t(\cW)$ be associated to a $\pm 1$-twisted Whitney disk $W_J$ such that $J$ has an $i$-labeled vertex which is not adjacent to its rooted edge. This means that $J$ contains a subtree of the form $((I_1,i),I_2)$ which corresponds to a Whitney disk $W_{((I_1,i),I_2)}$ supporting $W_J$, or possibly $W_{((I_1,i),I_2)}=W_J$. Do the $W_{(I_1,i)}$-Whitney move on $W_{I_1}$ as illustrated in Figure~\ref{collapsing-fig}. This eliminates the pair of intersections between $W_i$ and $W_{I_1}$ that were paired by $W_{(I_1,i)}$, and creates two new pairs of intersections $\{p^+,p^-\}$ and $\{q^+,q^-\}$
in $W_{I_1}\cap W_{I_2}$. %(Figure~\ref{?}). 
The $W_{(I_1,i)}$-Whitney move on $W_{I_1}$ also eliminates $W_{(I_1,i)}$ and all Whitney disks supported by $W_{(I_1,i)}$, including $W_J$. We will describe how to locally recover the desired trees:

The new intersection pairs  $\{p^+,p^-\}$ and $\{q^+,q^-\}$ admit Whitney disks $W^p_{(I_1,I_2)}$ and $W^q_{(I_1,I_2)}$, respectively, which are parallel copies of
the Whitney disk $W_{((I_1,i),I_2)}$ (Figure~\ref{collapsing-fig}).
Consider first the case that 
$W_{((I_1,i),I_2)}=W_J$. In this case each of the parallel copies $W^p_{(I_1,I_2)}$ and $W^q_{(I_1,I_2)}$ are also $\pm 1$-twisted, and there is a single intersection $r\in W^p_{(I_1,I_2)}\cap W^q_{(I_1,I_2)}$ with associated tree $t_r=\pm\langle (I_1,I_2),(I_1,I_2)\rangle$. So the effect on $t(\cW)$ is that $\pm J^\iinfty$ has been exchanged for $+I^\iinfty -I^\iinfty \pm\langle I,I \rangle$, where $I=(I_1,I_2)$ is the result of collapsing an $i$-labeled edge of $J=((I_1,i),I_2)$.

Now considering the case that 
$W_{((I_1,i),I_2)}$ supports $W_J$, the Whitney disks $W^p_{(I_1,I_2)}$ and $W^q_{(I_1,I_2)}$ are framed and disjoint since they are parallels of 
$W_{((I_1,i),I_2)}$. Each of $W^p_{(I_1,I_2)}$ and $W^q_{(I_1,I_2)}$ will inherit the intersections $W_{((I_1,i),I_2)}$ had with some $W_{I_3}$, and by taking parallels of the Whitney disks that were supported by $W_{((I_1,i),I_2)}$ all the way up to $W_J$ we see that $W^p_{(I_1,I_2)}$ and $W^q_{(I_1,I_2)}$ each support Whitney disks leading up to parallels $W_I^p$ and $W_I^q$ of $W_J$ whose trees $I$ are the result of collapsing an $i$-labeled vertex of $J$.
Since $W_I^p$ and $W_I^q$ are parallel copies of $W_J$, they are each $\pm 1$-twisted and have a single unpaired intersection $r=W_I^p\cap W_I^q$.
So the effect on $t(\cW)$ is that $\pm J^\iinfty$ has been exchanged for $+I^\iinfty -I^\iinfty \pm\langle I,I \rangle$, where $I$ is the result of collapsing an $i$-labeled edge of $J$.
\end{proof}

\begin{figure}[ht!]
        \centerline{\includegraphics[scale=.45]{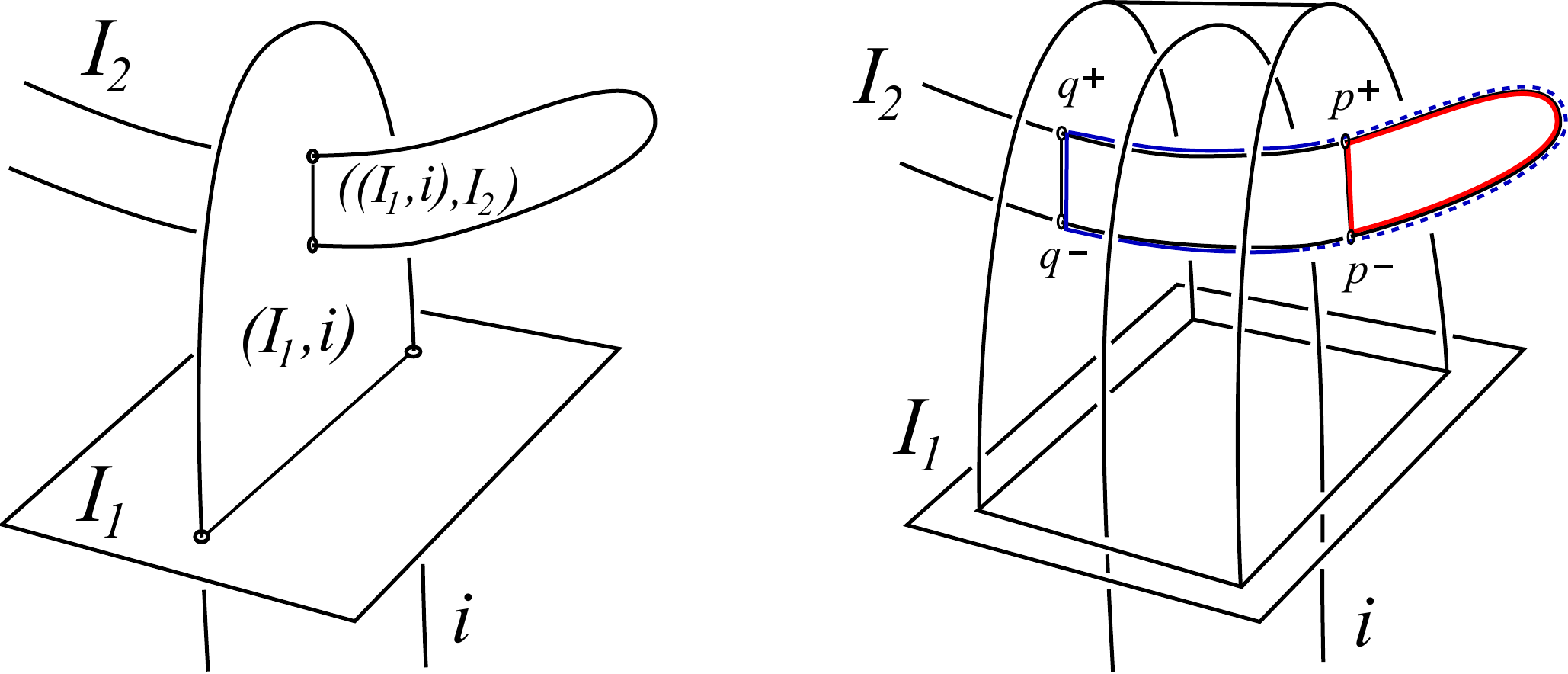}}
        \caption{Left: A $4$-ball neighborhood of the Whitney disk $W_{(I_1,i)}$ pairing $W_{I_1}\pitchfork W_i$, with $W_{(I_1,i)}\pitchfork W_{I_2}$ paired by $W_{((I_1,i),I_2)}$. Right: Doing the $W_{(I_1,i)}$ Whitney move on $W_i$ creates $\{p^+,p^-\}$ and $\{q^+,q^-\}$ in $W_{I_1}\cap W_{I_2}$.
The Whitney disk $W^p_{(I_1,I_2)}$ pairing $p^\pm$ is a trimmed parallel copy of $W_{((I_1,i),I_2)}$, with $\partial W^p_{(I_1,I_2)}$ shown in red.
The Whitney disk $W^q_{(I_1,I_2)}$ pairing $q^\pm$ is formed from a parallel copy of $W_{((I_1,i),I_2)}$, and $\partial W^q_{(I_1,I_2)}$ is indicated in blue.
The dashed blue subarc indicates that part of $W^q_{(I_1,I_2)}$ lies in a nearby `past or future' coordinate via the convention of considering $4$-space as $3$-space crossed with a time interval. In the case that $W_{((I_1,i),I_2)}$ equals the twisted Whitney disk $W_J$, then the figure is slightly inaccurate in that the pictured $W_{((I_1,i),I_2)}$ is framed.}
        \label{collapsing-fig}
\end{figure}        

%%%%%%%%%%%%%%%%%%%%%%%%%%%%%%%%%%%%%%%%%%

%\begin{rem}
%In \cite{ST3} the notion of \emph{non-repeating} Whitney towers, which is the case $k=1$ in Definition~\ref{def:k-rep-w-tower}, was shown to correspond to Milnor's  non-repeating ($1$-repeating) $\mu$ invariants which characterize link-homotopy to the unlink.
%\end{rem}
%%%%%%%%%%%%%%%%%%%%%%%%%%%%%%%%%%%%%%%%%%%%%%%%%%

\subsection{$k$-repeating twisted Whitney towers and twisted self $C_k$-concordance}\label{subsec:k-rep-w-towers-and-self-Ck-conc}

The following result covers the equivalence of the first two statements in Theorem~\ref{thm:twisted-self-Ck-equals-Milnor-Arf-equals-tower}:

\begin{prop}\label{prop:order-mk-1-W-tower-equals-self-Ck-concordance}
An $m$-component link $L\subset S^3$ bounds a $k$-repeating twisted Whitney tower of order $mk-1$ if and only if $L$ is twisted self $C_k$-concordant to the unlink.

\end{prop}

%For $k'>k$, self $C_{k'}$-equivalence implies self $C_k$-equivalence, and hence self $C_{k'}$-concordance implies self $C_k$-concordance. REF?
%

\begin{proof}
The ``if'' direction follows directly from Theorem~\ref{thm:clasper-forest-implies-w-tower-forest}: A twisted self $C_k$-concordance gives a twisted Whitney tower $\cW$ bounded by $L$ which is a $k$-repeating Whitney tower of order $n$ \emph{for all $n$}, since $t(\cW)$ consists of only mono-labeled trees of multiplicity $k+1$. 

For the ``only if'' direction, given a $k$-repeating twisted Whitney tower $\cW$ of order $mk-1$ bounded by $L$, Lemma~\ref{lem:reduce-to-mono-labeled} yields a $k$-repeating twisted Whitney tower $\cW'$ of order $mk-1$ bounded by $L$ such that $t(\cW')$ contains only mono-labeled trees of multiplicity $\geq k+1$. Now Theorem~\ref{thm:w-tower-forest-implies-clasper-forest}
implies that $L$ is self $C_k$-concordant to the unlink.
%If $t(\cW)=\emptyset$ contains no trees, then $\cW$
%is a collection of disjointly embedded slice disks for $L$ (perhaps after Whitney moves on clean framed Whitney disks in $\cW$) which can easily be turned into a concordance between $L$ and the unlink $U$. 
%
\end{proof}

Showing the equivalence of the second and third statements in Theorem~\ref{thm:twisted-self-Ck-equals-Milnor-Arf-equals-tower} will occupy the rest of this section.

%%%%%%%%%%%%%%%%%%%%%%%%%%%%%%%%%%%%%%%%%%

\subsection{Quick review of first non-vanishing Milnor invariants.}\label{subsec:intro-Milnor-review} 
Let $L\subset S^3$ be an $m$-component link with fundamental group $G=\pi_1(S^3\setminus L)$.
If the
longitudes of $L$ lie in the $(n+1)$-th term of the
lower central series $G_{n+1}$, then a choice of meridians induces an isomorphism
$
\frac{G_{n+1}}{G_{n+2}}\cong\frac{F_{n+1}}{F_{n+2}}
$
where $F=F(m)$ is the free group on $\{x_1,x_2,\ldots,x_m\}$.

Let $\sL=\sL(m)$ denote the free Lie algebra (over $\Z$) on generators $\{X_1,X_2,\ldots,X_m\}$. It is $\N$-graded, $\sL=\oplus_n \sL_n$, where the degree~$n$ part $\sL_n$ is the
additive abelian group of length $n$ brackets, modulo Jacobi
identities and self-annihilation relations $[X,X]=0$.
The multiplicative abelian group $\frac{F_{n+1}}{F_{n+2}}$ of
length $n+1$ commutators is isomorphic to
$\sL_{n+1}$, with $x_i$ mapping to  $X_i$ and commutators mapping to Lie brackets.

In this setting, denote by $l_i$ the image of the $i$-th longitude in  $\sL_{n+1}$ under the above isomorphisms and
define the \emph{order $n$ Milnor invariant}
$\mu_n(L)$ by
$$
\mu_n(L):=\sum_{i=1}^m X_i \otimes l_i \in \sL_1 \otimes
\sL_{n+1}.
$$
This definition of $\mu_n(L)$ is the first non-vanishing ``total'' Milnor invariant of order $n$, and corresponds to {\em all} Milnor invariants of \emph{length} $n+2$ in the original formulation of \cite{M1,M2}. 
The original $\bar{\mu}$-invariants are computed from the longitudes via the Magnus expansion as integers modulo indeterminacies coming from invariants of shorter length. Since we will only be concerned with first non-vanishing $\mu$-invariants, we do not use the ``bar'' notation $\bar{\mu}$.

It turns out that  $\mu_n(L)$ lies in the kernel $\sD_n$
of the bracket map $\sL_1 \otimes \sL_{n+1}\rightarrow \sL_{n+2}$ 
(by ``cyclic symmetry'' \cite{FT2}).

\subsection{Intersection invariants for twisted Whitney towers}\label{subsec:W-tower-int-trees}

The intersection forest $t(\cW)$ (Definition~\ref{subsec:int-forests}) of an order $n$ twisted Whitney tower $\cW$ determines an intersection invariant $\tau_n^\iinfty(\cW)$ which represents the obstruction to the existence of an order $n+1$ twisted Whitney tower on the same order $0$ surfaces and takes values in a finitely generated abelian group $\cT_n^\iinfty$ as we briefly describe next. For details see \cite{CST1,ST2}.

Denote by $\cT=\cT(m)$ the quotient of the free abelian group on framed trees (as in section~\ref{sec:trees}) by the following local \emph{antisymmetry} (AS) and \emph{Jacobi} (IHX) relations: 
\begin{figure}[h]
\centerline{\includegraphics[scale=.65]{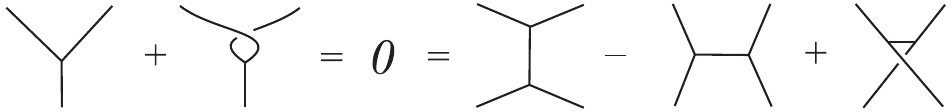}}
         \label{fig:ASandIHXtree-relations}
\end{figure}

Since the AS and IHX relations are homogeneous with respect to order, $\cT$ inherits a grading $\cT=\oplus_n\cT_n$, where $\cT_n=\cT_n(m)$ is the free abelian group on order $n$ trees, modulo AS and IHX relations.

As a consequence of the AS relations the signs of the framed trees in $t(\cW)$ only depend on the orientation of the order zero surfaces after mapping to $\cT_n$ \cite[Sec.2.5]{CST2}.

For a twisted Whitney tower $\cW$ of order $n$, the twisted intersection invariant $\tau_n^\iinfty(\cW)$ 
is defined by taking the signed sum of all order $n$ framed trees and order $\frac{n}{2}$ $\iinfty$-trees in the intersection forest $t(\cW)$ in a finitely generated abelian group $\cT_n^\iinfty$
which is defined as follows: 

\begin{defn}[\cite{CST1}]\label{def:T-infty}
In odd orders, the group $\cT^{\iinfty}_{2j-1}$ is the quotient of $\cT_{2j-1}$ by the \emph{boundary-twist relations}: 
\[
 i\,-\!\!\!\!\!-\!\!\!<^{\,J}_{\,J}\,\,=\,0
\] 
where $J$ ranges over all order $j-1$ subtrees. These relations correspond to the fact that
via the boundary-twisting operation \cite[Fig.18]{CST1} any order $2j-1$ tree $ i\,-\!\!\!\!\!-\!\!\!<^{\,J}_{\,J}\,\in t(\cW)$ can be replaced by the $\iinfty$-tree $ i\,-\!\!\!\!\!-\!\!\!<^{\,J}_{\,\iinfty}$ of order $j$ at the cost of only creating trees of order $\geq 2j$ in $t(\cW)$.

In even orders, the group $\cT^{\iinfty}_{2j}$ is the quotient of the free abelian group on trees of order $2j$ 
and $\iinfty$-trees of order $j$ by the following relations:
\begin{enumerate}
     \item AS and IHX relations on order $2j$ trees
   \item \emph{symmetry} relations: $(-J)^\iinfty = J^\iinfty$
  \item \emph{twisted IHX} relations: $I^\iinfty=H^\iinfty+X^\iinfty- \langle H,X\rangle $
   \item {\em interior-twist} relations: $2\cdot J^\iinfty=\langle J,J\rangle $
\end{enumerate}
The AS and IHX relations are as pictured above, but they only apply to framed trees (not to $\iinfty$-trees). 
The \emph{symmetry relation} corresponds to the fact that the relative 
Euler number of a Whitney disk is independent of its orientation, with the minus sign denoting that the cyclic edge-orderings at the trivalent vertices of $-J$ differ from those of $J$ at an odd number of vertices.
\end{defn}
As explained in \cite{CST1}, the \emph{twisted IHX relation} corresponds to the effect of performing a Whitney move in the presence of a twisted Whitney disk, and the \emph{interior-twist relation} corresponds to the fact that creating a 
$\pm1$ self-intersection
in a Whitney disk changes its twisting by $\mp 2$.
The IHX relations on framed trees can also be realized geometrically, and we have the following obstruction theory for order $n$ twisted Whitney towers:
\begin{thm}[Thm.1.9 of \cite{CST1}]\label{thm:twisted-order-raising}
A link $L$ bounds an order $n$ twisted Whitney tower $\cW$ with $\tau_n^\iinfty(\cW)=0\in\cT_n^\iinfty$ if and only if $L$ bounds an order $n+1$ twisted Whitney tower.$\hfill\square$
\end{thm}
An outline of the proof of this theorem will be given during the proof-sketch of its $k$-repeating analogue 
Theorem~\ref{thm:k-rep-twisted-order-raising} in section~\ref{subsec:twisted-invariant-def}.

\subsection{The summation maps $\eta_n$}\label{subsec:eta-map}
The connection between $\tau^\iinfty_n(\cW)$ and $\mu_n(L)$ is via a homomorphism $\eta_n : \cT^\iinfty_n \to \sD_n$, which is most easily described by regarding rooted trees of order $n$ as elements of $\sL_{n+1}$ in the usual way:
For $v$ a univalent vertex of an order $n$ framed tree $t$, denote by $B_v(t)\in\sL_{n+1}$ the Lie bracket of generators $X_1,X_2,\ldots,X_m$ determined by the formal bracketing of indices
which is gotten by considering $v$ to be a root of $t$.

\begin{defn}\label{def:eta}
Denoting the label of a univalent vertex $v$ by $\ell(v)\in\{1,2,\ldots,m\}$, the
map $\eta_n:\cT^\iinfty_n\rightarrow \sL_1 \otimes \sL_{n+1}$
is defined on generators by
$$
\eta_n(t):=\sum_{v\in t} X_{\ell(v)}\otimes B_v(t)
\quad \, \,
\mbox{and}
\quad \, \,
\eta_n(J^\iinfty):= \frac{1}{2}\,\eta_n(\langle J,J \rangle).
$$
The first sum is over all univalent vertices $v$ of $t$, and the second expression lies in $\sL_1 \otimes \sL_{n+1}$ 
because the coefficients of $\eta_n(\langle J,J \rangle)$ are even. Here $J$ is a rooted tree of order $j$ for $n=2j$.
\end{defn}
%See \cite[Sec.1.3]{CST1} for examples of the $\eta_n$ map.

The image of $\eta_n$ is contained in the bracket kernel $\sD_n<\sL_1 \otimes \sL_{n+1}$ 
(see e.g.~\cite[Lem.32]{CST2}
).

\begin{thm}[\cite{CST1}]\label{thm:Milnor invariant}
If $L$ bounds a twisted Whitney tower $\cW$ of order $n$, then the order $q$ Milnor invariants $\mu_q(L)$ vanish for $q<n$, and
\[
\mu_n(L) = \eta_n \circ\tau^\iinfty_n(\cW) \in \sD_n.
\]
\end{thm}

%\subsection{The twisted Whitney tower filtration} \label{subsec:twisted-W-tower-filtration}
\subsection{Classification of the twisted Whitney tower filtration}\label{subsec:twisted-filtration-classification}
%
%
%REWRITE/ONLY DO TWISTED...
%
%\begin{figure}[ht!]
%        \centerline{\includegraphics[width=160mm]{classification-pic}}
%        \caption{}
%        \label{CLASSIFICATION-PIC}
%\end{figure}
%

Following \cite{CST1}, the set of $m$-component framed links in $S^3$ which bound order $n$ twisted Whitney towers in $B^4$ is denoted by 
$\bW^\iinfty_n$; and the quotient of $\bW^\iinfty_n$ by the equivalence relation of order $n+1$ \emph{twisted Whitney tower concordance} is 
denoted by $\W^\iinfty_n$. Here a twisted Whitney tower concordance between two links is a twisted Whitney tower on a singular concordance of immersed annuli in $S^3\times I$ between the links. 
As a consequence of the twisted Whitney tower obstruction theory 
(Theorem~\ref{thm:twisted-order-raising}), the component-wise band-sum operation makes $\W^\iinfty_n$ a finitely generated abelian group. (See \cite{CST1} for details.)

The quotient $\W^\iinfty_n$  
is the {\em associated graded} of the filtration $\bW^\iinfty_n$ in the sense that $L\in \bW^\iinfty_{n+1}$ if and only if $L\in\bW^\iinfty_n$ and $[L]=0\in\W^\iinfty_n$, with $0$ corresponding to the unlink.

Using Cochran's Bing-doubling techniques (as in section~\ref{subsec:Bing-along-trees}), we constructed in \cite[sec.3.2]{CST1} twisted \emph{realization} epimorphisms
$$
R^\iinfty_n : \cT^\iinfty_n \sra\W^\iinfty_n
$$
which send $g\in\cT^{\iinfty}_n$ to the equivalence class of links bounding an order $n$ twisted Whitney tower $\cW$ with $\tau^{\iinfty}_n(\cW)=g$.

From Theorem~\ref{thm:Milnor invariant} there is a commutative triangle diagram of epimorphisms:

\begin{equation}\label{eqn:triangle}
    \xymatrix{
\cT^\iinfty_n \ar@{->>}[r]^{R_n^\iinfty} \ar@{->>}[rd]_{\eta_n} & \W^\iinfty_n \ar@{->>}[d]^{\mu_n}\\
& \sD_{n}
}
\tag{$\bigtriangledown$}
\end{equation}

The following partial classification of $\W^\iinfty_n$ is a consequence of our proof \cite{CST3} of 
a combinatorial conjecture of J. Levine \cite{L2}:
\begin{thm}[\cite{CST1}] \label{thm:twisted-three-quarters-classification} The maps $\eta_n:\cT^\iinfty_n \to \sD_n$ are isomorphisms for $n\equiv 0,1,3\,\mod 4$.
As a consequence, both the Milnor invariants $\mu_n\colon \W^\iinfty_n\to \sD_n$  and the twisted realization maps $R^\iinfty_n : \cT^\iinfty_n \to\W^\iinfty_n$ are isomorphisms for these orders.
\end{thm}

Since $\sD_n$ is a free abelian group of known rank for all $n$ \cite{O}, this gives a complete computation of $\W^\iinfty_n$ in three quarters of the cases.

The following result computes the kernel of $\eta_n$ for all $n\equiv 2\mod 4$:
\begin{prop}[\cite{CST1}]\label{prop:kerEta4j-2}
The map sending $1\otimes J $ to $ \iinfty-\!\!\!\!-\!\!-\!\!\!<^{\,J}_{\,J}\,\,\in\cT^\iinfty_{4j-2}$ 
for rooted trees $J$ of order $j-1$ defines an isomorphism $\Z_2 \otimes \sL_j\cong\Ker(\eta_{4j-2}:\cT_{4j-2}^\iinfty\to\sD_{4j-2})$.
\end{prop}

It follows from the above commutative triangle diagram \eqref{eqn:triangle} that $\Z_2 \otimes \sL_j$ is also an upper bound on
the kernels of the epimorphisms $R^\iinfty_{4j-2}$ and $\mu_{4j-2}$, and the calculation of $\W^\iinfty_{4j-2}$ is completed by \emph{higher-order Arf invariants} defined on the kernel of 
$\mu_{4j-2}$, as we describe next.

%%%%%%%%%%%%%%%%%%%%%%%%%%%%%%%%%%%%%%%%%%%%%%%%%%%%%%%%%%%%%%%%%%%%%%%%%%%%%%%

\subsection{Higher-order Arf invariants}\label{subsec:intro-higher-order-arf}

Let $\sK^\iinfty_{4j-2}$ denote the kernel of $\mu_{4j-2}: \W^\iinfty_{4j-2} \sra \sD_{4j-2}$. It follows from the triangle diagram \eqref{eqn:triangle} and Proposition~\ref{prop:kerEta4j-2} above that mapping $1\otimes J$ to 
$R^\iinfty_{4j-2}( \iinfty-\!\!\!\!\!-\!\!\!<^{\,J}_{\,J}\,\,)$ induces a surjection $\alpha^\iinfty_j: \Z_2 \otimes \sL_j
\sra \sK^\iinfty_{4j-2}$, for all $j\geq 1$.
Denote by $\overline{\alpha^\iinfty_{j}}$ the induced isomorphism on $(\mathbb Z_2\otimes {\sf L}_{j})/\Ker \alpha^\iinfty_{j}$.

\begin{defn}\label{def:Arf-j}  
The \emph{higher-order Arf invariants} are defined by
$$
\Arf_{j}:=(\overline{\alpha^\iinfty_{j}})^{-1}:\sK^\iinfty_{4j-2}\to(\mathbb Z_2\otimes {\sf L}_{j})/\Ker \alpha^\iinfty_{j}.
$$
\end{defn}

So from the triangle diagram \eqref{eqn:triangle}, Theorem~\ref{thm:twisted-three-quarters-classification}, Proposition~\ref{prop:kerEta4j-2} and Definition~\ref{def:Arf-j} we see that the groups $\W^\iinfty_n$ are computed by the Milnor and higher-order Arf invariants:
\begin{cor}[\cite{CST1}]\label{cor:intro-mu-arf-classify-twisted} 
The groups ${\sf W}^\iinfty_{n}$ are classified by Milnor invariants $\mu_n$ and, in addition, higher-order Arf invariants $\Arf_j$ for $n=4j-2$.

In particular, a link bounds an order $n+1$ twisted Whitney tower if and only if its Milnor invariants and higher-order Arf invariants vanish up to order $n$.

\end{cor}
For the case $j=1$, the classical Arf invariants of the link components correspond to $\Arf_1$ \cite{CST2},
but it remains an open problem whether $\Arf_j$ is non-trivial for any $j>1$. 
The links $R^\iinfty_{4j-2}( \iinfty-\!\!\!\!\!-\!\!\!\!<^{\,J}_{\,J}\,\,)$ realizing the image of $\Arf_{j}$ can all be constructed as internal band sums of iterated Bing doubles of knots having non-trivial classical Arf invariant \cite{CST2}. Such links (which are all boundary links) are known not to be slice
by work of J.C. Cha \cite{Cha}, providing evidence in support of the following Conjecture:

\begin{conj}[\cite{CST1}]\label{conj:Arf-j}
 $\Arf_{j}:\sK^\iinfty_{4j-2}\to\mathbb Z_2\otimes {\sf L}_{j}$ are isomorphisms for all $j$.
\end{conj}

Conjecture~\ref{conj:Arf-j} would imply that  $\W_{4j-2}^\iinfty\cong \cT^\iinfty_{4j-2} \cong (\Z_2 \otimes \sL_j) \oplus \sD_{4j-2}$ where the second isomorphism (is non-canonical and)
already follows from Proposition~\ref{prop:kerEta4j-2}.

\subsection{The $k$-repeating (labeled) free Lie algebra}\label{subsec:k-free-lie}
Recall from section~\ref{subsec:intro-Milnor-review} that $\sL=\oplus_n \sL_n$ denotes the free $\Z$-Lie algebra on $\{X_1,X_2,\ldots,X_m\}$, 
which we identify with the additive abelian group on rooted trees
modulo IHX and self-annihilation relations,
with the subgroup $\sL_n$ of degree $n$ brackets corresponding to the subgroup of order $n-1$ rooted trees.
Here the IHX relation corresponds to the Jacobi identity, and the self-annihilation relations set a tree equal to zero if it contains a sub-tree of the form $(J,J)$.

Define the \emph{$k$-repeating free Lie algebra} $\sL^k=\oplus_n \sL^k_n$ on the $X_i$ 
to be the quotient of
$\sL$ by the relations which set equal to zero all rooted trees of multiplicity $>k$.
So $\sL^k_n$ is the subgroup of $\sL_n$ spanned by order $n-1$ rooted trees (degree $n$ brackets) of multiplicity $\leq k$. 

%Note that $\sL^k_n=0$ for $n>km-1$, since an order $n$ tree with an un-labeled root univalent vertex has $n+1$ labeled univalent vertices.

Elaborating on the identification of $\sL$ with rooted trees modulo IHX and self-annihilation, we now identify $\sL_1\otimes\sL_{n+1}$ with the abelian group on order $n$ rooted trees with \emph{roots labeled} by an index from $\{1,2,\ldots,m\}$,
modulo IHX and self-annihilation relations. We refer to a rooted tree whose root vertex is labeled by an index from $\{1,2,\ldots,m\}$ as a \emph{root-labeled} tree.
Define $(\sL_1\otimes\sL_{n+1})^k$ to be the quotient
of $\sL_1\otimes\sL_{n+1}$ by the relations which set equal to zero all root-labeled trees of multiplicity $>k$.

\subsection{$k$-repeating quotients}\label{subsec:k-nilpotent-quotients}
For a fixed set of generators of a group $G$, define the \emph{$k$-repeating quotient} $G^k$ to be the quotient of $G$ by the subgroup generated by arbitrary length iterated commutators 
of the generators with multiplicity $>k$, which by basic properties of commutators is a normal subgroup of $G$.

Note that for $F$ the free group on $\{x_1,x_2,\ldots,x_m\}$ the (multiplicative) lower central quotient $F_{n+1}^k/F^k_{n+2}$ of the $k$-repeating quotient $F^k$ is isomorphic to the (additive) subgroup $\sL^k_{n+1}<\sL_{n+1}$.

\subsection{$k$-repeating Milnor invariants}\label{subsec:k-rep-milnor-invariants}

%MAYBE BETTER TO CALL THEM $\leq k$-repeating Milnor invariants?
%
%IDEA, AS WITH k-REPEATING WHITNEY TOWERS, IS THAT COMMUTATORS OF MULTIPLICITY $>k$ CAN BE IGNORED...

In this section we define the first non-vanishing total $k$-repeating order $n$ Milnor invariant $\mu^k_n(L)$, 
which, roughly speaking, differs from $\mu_n(L)$ as in section~\ref{subsec:intro-Milnor-review} in that $\mu^k_n(L)$
ignores all iterated commutators of multiplicity $>k$.

For a link $L=\cup_{i=1}^m L_i\subset S^3$ with $G:=\pi_1(S^3\setminus L)$ we have the following presentation %(see e.g.~\cite{M2}) 
for the $(n+2)$th lower central quotient $G/G_{n+2}$: 
$$
G/G_{n+2}\cong \langle x_1,x_2,\ldots,x_m\,  | \,  [x_i,w_i], F_{n+2} \,\rangle
$$
where $x_i$ is represented by a meridian to $L_i$, with $w_i$ a word in the generators representing a $0$-parallel longitude of $L_i$. 
%Here each relator $[x_i,w_i]$ is the $2$-cell attaching map of a torus Alexander dual to $L_i$, and these tori generate the integral second homology $H_2(S^3\setminus L)$.

\textbf{Assumption:} We assume that each $w_i$ is contained in the subgroup of $G/G_{n+2}$ generated by length $(n+1)$ commutators and arbitrary length commutators $c$ such that $r([x_i,c])>k$.

Under this assumption $(G/G_{n+2})^k$
is isomorphic to the free $k$-repeating nilpotent quotient $(F/F_{n+2})^k\cong F^k/F^k_{n+2}$. Via this isomorphism we now consider $w_i\in F^k/F^k_{n+2}$, and denote by 
$l_i^k\in\sL^k_{n+1}<\sL_{n+1}$ the image of $w_i$ in $F_{n+1}^k/F^k_{n+2}\cong\sL^k_{n+1}<\sL_{n+1}$.
The \emph{total $k$-repeating order $n$ Milnor invariant}
$\mu^k_n(L)$ is then defined by
$$
\mu^k_n(L):=\sum_{i=1}^m X_i \otimes l^k_i \in (\sL_1 \otimes
\sL_{n+1})^k.
$$
Note that the vanishing of $\mu^k_q(L)$ for all $q<n$ implies the above Assumption.
This total $k$-repeating order $n$ Milnor invariant $\mu^k_n(L)$ determines all Milnor invariants $\mu_{\cI}(L)$ of \emph{length} $n+2$ where the multi-index $\cI$ has multiplicity $\leq k$. 

The image of $\mu^k_n(L)$ is contained in the kernel $\sD^k_n<\sD_n$ of the (restricted) bracket map $(\sL_1 \otimes
\sL_{n+1})^k\to\sL^k_{n+2}$.

In the case $k=1$, which corresponds to Milnor's original ``non-repeating'' link homotopy $\mu$-invariants, the longitudinal elements $l_i^1<\sL^{1}_{n+1}$ are length $n+1$ commutators in the \emph{Milnor group} of the sublink of $L$ gotten by deleting the $i$th component \cite{M1}.

\subsection{$k$-repeating nilpotent quotients and twisted Whitney towers}\label{subsec:k-rep-w-tower-complement}

\begin{lem}\label{lem:k-rep-w-tower-complement}
Let $L=\cup_{i=1}^m\subset S^3$ be a link bounding a $k$-repeating order $n$ twisted Whitney tower $\cW\subset B^4$, with
$G=\pi_1(S^3\setminus L)$ and $\cG=\pi_1(B^4\setminus \cW)$. 
Then the following three statements hold:
\begin{enumerate}
\item\label{item:H1-free}
The integral first homology $H_1(B^4\setminus \cW)$ is isomorphic to $\oplus_{i=1}^m\Z$, with generators $ x_1,x_2,\ldots,x_m$, 
where each $x_i$ is represented by a meridian to the order $0$ disk $W_i$ with $\partial W_i=L_i$.

\item\label{item:nilpotent-presentation}
The nilpotent quotient $\cG/\cG_{n+2}$ has the presentation:
$$
\cG/\cG_{n+2}\cong \langle x_1,x_2,\ldots,x_m\,  | \,  c_1,c_2,\ldots, c_M, F_{n+2} \,\rangle
$$
where the $x_i$ are as in \emph{statement~(\ref{item:H1-free})}, and each word $c_j$ is an iterated commutator in the $x_i$ of multiplicity $r(c_j)>k$. 
%Specifically, for each $j$ the rooted tree corresponding to $c_j$ is gotten from a framed tree $t_p\in t(\cW)$ by attaching a rooted edge to the interior of an edge of $t_p$, or is gotten from a $\iinfty$-tree $J^\iinfty\in t(\cW)$ by attaching a rooted edge to the interior of an edge of $\langle J,J\rangle$.

\item\label{item:inclusion-iso}
The inclusion 
$S^3\setminus L \hra B^4\setminus \cW$ induces isomorphisms on the $k$-repeating nilpotent quotients:
$$G^k/G^k_{n+2}\cong \cG^k/\cG^k_{n+2}\cong F^k/F^k_{n+2}.$$

\end{enumerate}

\end{lem}

\begin{proof}
Statement~(\ref{item:H1-free}):
Since $H_1(B^4)=0$, any $1$-cycle in $H_1(B^4\setminus \cW)$ can be represented by a union of circles, each of which is a meridian to a Whitney disk or an order zero disk in $\cW$. Each meridian to a Whitney disk is null-homologous since it bounds a punctured Clifford torus displaying it as a commutator in $\pi_1(B^4\setminus \cW)$ of meridians to the two lower order surfaces paired by the Whitney disk, see \cite[Lem.3.6]{Cha2} or \cite[Fig.14]{CST2}. So $H_1(B^4\setminus \cW)$ is generated by $\{x_1,x_2,\ldots,x_m\}$, where each $x_i$ a meridian to the order zero surface 
$W_i$.

Any relation in $H_1(B^4\setminus \cW)$ is represented by a surface $S$ with $\partial S$ a union of meridians to the $W_i$. But since any closed surface in $B^4$ has zero algebraic intersection with each $W_i$, the relation represented by $S$ restricts for each $i$ to a trivial relation in $x_i$, so 
$H_1(B^4\setminus \cW)\cong\oplus_{i=1}^m\Z$.

Statement~(\ref{item:nilpotent-presentation}):
It suffices (e.g.~\cite[Lem.13]{Kr2}) to show that
the commutators $c_j$ are $2$-cell attaching maps for surfaces representing a generating set for 
$H_2(B^4\setminus \cW)$ which are Alexander dual to a generating set for $H_1(\cW)$.

The generators of $H_2(B^4\setminus \cW)$ are tori which are Alexander duals to the sheet-changing loops in $\cW$ which generate $H_1(\cW)$ (see \cite[Prop.25]{ST3}), and these tori are either Clifford tori around un-paired intersections, or normal circle-bundles over circles around boundary-arcs of Whitney disks (these circles are called ``ovals'' in \cite[Prop.25]{ST3}). The attaching maps for these generators are iterated commutators in the $x_i$ which are determined by $t(\cW)$ as follows.
Let $e$ be an edge of a framed tree $t_p\in t(\cW)$, and let $\Sigma$ be a torus dual to a sheet-changing loop through a transverse intersection point in $\cW$ corresponding to $e$. (If $p\in e$ then $\Sigma$ is a Clifford torus around $p$, and if $p\notin e$ then $\Sigma$ is a normal circle-bundle over an oval around the boundary of a Whitney disk where $e$ changes sheets.)  
By Whitney tower-grope duality \cite[Prop.25]{ST3} (see also \cite[Lem.3.11]{Cha2} and \cite[Lem.33]{CST2}), the attaching map of the $2$-cell of $\Sigma$ is the iterated commutator in the $x_i$ corresponding to the rooted tree gotten by
attaching a new rooted edge to $e$ and creating a new trivalent vertex in what used to be an interior point of $e$.
If $t_p$ has order $n$, then this commutator will be length $n+2$ so the attaching map will be trivial in $\cG/\cG_{n+2}$.
If $t_p$ has order $<n$, then $t_p$ and the resulting attaching map commutator must have multiplicity $>k$. 
The same argument applied to the tree $\langle J,J \rangle$ describes the attaching maps for tori corresponding to edges in a $\iinfty$-tree in $J^\iinfty\in t(\cW)$ (see proof of \cite[Lem.3.11(2)]{Cha2}).

Statement~(\ref{item:inclusion-iso}):
By the previous statement, the inclusion $S^3\setminus L\hra B^4\setminus \cW$ induces an isomorphism on first homology, and the kernel of the induced surjection $G/G_{n+2}\sra \cG/\cG_{n+2}$ is contained in the inverse image of the iterated commutator relations $c_j$ in the presentation of Statement~(\ref{item:nilpotent-presentation}). Since the map on first homology sends the $i$th meridian to the $i$th meridian, the inverse image of each $c_j$ has multiplicity $>k$. It follows that the induced map on $k$-repeating quotients is an isomorphism 
$G^k/G^k_{n+2}\cong(G/G_{n+2})^k\cong (\cG/\cG_{n+2})^k\cong\cG^k/\cG^k_{n+2}$.
The isomorphism $\cG^k/\cG^k_{n+2}\cong F^k/F^k_{n+2}$ follows directly from the presentation of Statement~(\ref{item:nilpotent-presentation}).
\end{proof}

\subsection{$k$-repeating intersection invariants for twisted Whitney towers}\label{subsec:twisted-invariant-def}

Recall the definition of $\cT^\iinfty_n$ from Definition~\ref{def:T-infty} of section~\ref{subsec:W-tower-int-trees}, as well as the definitions of multiplicities for framed trees and $\iinfty$-trees given in section~\ref{subsec:multiplicities}.

\begin{defn}\label{def:k-rep-infty-tree-groups}
For each $n$, the group ${\cT^k_n}^\iinfty$ is defined to be the subgroup of $\cT^\iinfty_n$ spanned by order $n$ trees of multiplicity $\leq k$ and order $n/2$ $\iinfty$-trees of multiplicity $\leq k$.
\end{defn}
Since the relations in $\cT^\iinfty_n$ are homogeneous with respect to tree multiplicities, ${\cT^k_n}^\iinfty$ is a direct summand of  $\cT^\iinfty_n$.

\begin{defn}\label{def:k-rep-tau-infty}
The \emph{$k$-repeating order $n$ intersection invariant}
${\tau^k_n}^{\iinfty}(\cW)$ of a $k$-repeating order
$n$ twisted Whitney tower $\cW$ is defined to be
$$
{\tau^k_n}^{\iinfty}(\cW):=\sum \epsilon_p\cdot t_p + \sum \omega(W_J)\cdot J^\iinfty\in{\cT^k_n}^\iinfty
$$
where the first sum is over all order $n$ intersections $p$ such that $r(t_p)\leq k$, and the second sum is over all order $n/2$
Whitney disks $W_J$ with twisting $\omega(W_J)\in\Z$ such that $r(J^\iinfty)\leq k$.

\end{defn}

We have the following $k$-repeating version of the obstruction theory described in Theorem~\ref{thm:twisted-order-raising}: 
\begin{thm}\label{thm:k-rep-twisted-order-raising}
A link $L$ bounds a $k$-repeating order $n$ twisted Whitney tower $\cW$ with ${\tau^k_n}^\iinfty(\cW)=0\in{\cT^k_n}^\iinfty$ if and only if $L$ bounds a $k$-repeating order $n+1$ twisted Whitney tower.
\end{thm}
\begin{proof}
The ``if'' direction holds since any $k$-repeating order $n+1$ twisted Whitney tower is also a $k$-repeating order $n$ twisted Whitney tower.

The ``only if'' direction follows from the observation that the order-raising constructions in the proof of Theorem~\ref{thm:twisted-order-raising} can be applied to all trees of multiplicity $\leq k$ in $t(\cW)$, while ignoring the presence of any lower-order trees in $t(\cW)$ which have multiplicity $>k$, as we summarize here:

%Recall that we can assume that $\cW$ is split (Lemma~\ref{lem:split-w-tower}). %and contains $t(\cW)$ as a subset. 
The condition ${\tau^k_n}^\iinfty(\cW)=0\in{\cT^k_n}^\iinfty$
means that the trees in $t(\cW)$ of multiplicity $\leq k$ represent $0\in\cT_n^\iinfty$, and any framed trees of order $<n$ or $\iinfty$-trees of order $<n/2$ in $t(\cW)$ must have multiplicity $>k$.

There are three main steps to the proof of Theorem~\ref{thm:twisted-order-raising} in \cite[Sec.4.1]{CST1}, using also \cite[Sec.4]{ST2} and \cite{CST}.
First, controlled modifications of $\cW$ realizing the relations in $\cT_n^\iinfty$ are used to arrange that the order $n$ trees and order $n/2$ trees in $t(\cW)$ all occur in oppositely-signed ``algebraically canceling'' pairs (see the start of section~4 of \cite{CST1}). Since the relations are homogeneous in multiplicities, this step does not create any trees of smaller (or greater) multiplicities and is supported away from all previously existing trees.
In our current $k$-repeating setting, all trees of multiplicity $\leq k$ can be paired while any trees of multiplicity $>k$ do not need to be paired. 
Secondly, the paired trees are converted into pairs of ``simple'' (right- or left-normed) trees by IHX constructions which are again multiplicity-preserving.
In both of these first two steps all intersections are completely controlled, and the constructions can be assumed to be supported away from any trees of multiplicity $>k$. %framed trees of order $<n$ and $\iinfty$-trees of order $<n/2$. 

The final third step converts algebraic cancellation into ``geometric cancellation'', which means that a new layer of order $n+1$ Whitney disks can be constructed for the pairs of order $n$ intersections, and the pairs of order $n/2$ twisted Whitney disks can be combined into single framed Whitney disks.
In this step the new layer of order $n+1$ Whitney disks have uncontrolled intersections, but all of these new intersections are of order $\geq n+1$. And the construction combining the twisted Whitney disk pairs into framed Whitney disks (Figures 21--22 in \cite[Sec.4.1]{CST1}) creates only new twisted Whitney disks of order $>n/2$, which are supported near the original twisted Whitney disk pairs, along with intersections of order $\geq n$ among these new twisted Whitney disks. 
\end{proof}

\subsection{The $k$-repeating twisted Whitney tower filtration}\label{subsec:k-rep-twisted-filtration}
Denote by ${\bW_n^k}^\iinfty$ the set of $m$-component framed links in $S^3$ which bound $k$-repeating order $n$ twisted Whitney towers in $B^4$; and let 
${\W_n^k}^\iinfty$ denote the quotient of ${\bW_n^k}^\iinfty$ by the equivalence relation of \emph{$k$-repeating order $n+1$ twisted Whitney tower concordance}.
 As a consequence of the $k$-repeating twisted Whitney tower obstruction theory (Theorem~\ref{thm:k-rep-twisted-order-raising}), %and its extension to immersed concordances
 the component-wise band-sum operation makes ${\W_n^k}^\iinfty$ into a finitely generated abelian group, with a surjective $k$-repeating twisted realization map ${R^k_n}^\iinfty : {\cT_n^k}^\iinfty \to{\W_n^k}^\iinfty$. This follows from the $k$-repeating analogue of \cite[Sec.3]{CST1}.

\subsection{The $k$-repeating summation maps $\eta^k_n$}\label{subsec:k-rep--eta-map}
The $k$-repeating summation maps $\eta^k_n:{\cT^k_n}^\iinfty\rightarrow \sD^k_n$ are defined as the restrictions
of the maps $\eta_n:\cT_n^\iinfty\rightarrow \sD_n$ of Definition~\ref{def:eta} in section~\ref{subsec:eta-map}.

The following $k$-repeating version of Theorem~\ref{thm:Milnor invariant} will yield in the subsequent section a $k$-repeating analogue of the diagram \eqref{eqn:triangle} from section~\ref{subsec:twisted-filtration-classification}:

\begin{thm}\label{thm:k-repeating-mu} 
If $L$ bounds a
$k$-repeating twisted Whitney tower $\cW$ of order $n$, then $\mu^k_q(L)=0$ for $q<n$, and
\[
\mu^k_n(L)=\eta^k_n({\tau^k_n}^\iinfty(\cW))  \in \sD^k_n.
\]
\end{thm}

\begin{proof}
For $G=\pi_1(S^3\setminus L)$ and $\cG=\pi_1(B^4\setminus \cW)$,  Lemma~\ref{lem:k-rep-w-tower-complement} gives isomorphisms $G^k/G^k_{n+2}\cong\cG^k/\cG^k_{n+2}\cong F^k/F^k_{n+2}$, with the first isomorphism induced by inclusion.
Then by (the proof of) \cite[Lem.3.8, Lem.3.9]{Cha2} the link longitudes are represented by the products of iterated commutators of meridians corresponding to the image under the $\eta_n^k$-map of the trees in $t(\cW)$ representing ${\tau^k_n}^\iinfty(\cW)$. 
%(The proof from \cite[Thm.6]{CST2} accomplishes the analogous result after converting $\cW$ to a twisted capped grope.)
Since $\cW$ is $k$-repeating order $n$ it follows that
the Assumption of section~\ref{subsec:k-rep-milnor-invariants} is satisfied, so $\mu^k_q(L)=0$ for $q<n$, and each $i$th longitude maps to $l_i^k\in\sL^{k,i}_{n+1}$.

To see that $\mu^k_n(L)=\eta^k_n({\tau^k_n}^\iinfty(\cW))$, observe that the calculation given in \cite[Lem.3.9]{Cha2} of the link longitudes as iterated commutators of meridians determined by the trees in  
$t(\cW)$
holds in $\cG/\cG_{n+2}$ even when $\cW$ is a $k$-repeating order $n$ twisted Whitney tower, rather than an ordinary order $n$ Whitney tower as considered in \cite{Cha2}. So in the current $k$-repeating setting the presence of framed trees which have both multiplicity $>k$ and order $<n$, and $\iinfty$-trees which have 
both multiplicity $>k$ and order $<\frac{n}{2}$, only contribute factors in $\cG/\cG_{n+2}$ to the longitudes which map trivially to $l_i^k\in\sL^{k,i}_{n+1}<\sL^k_{n+1}<\sL_{n+1}$, while the order $n$ framed trees and order $n/2$ $\iinfty$-trees contribute longitude factors in $\cG_{n+1}/\cG_{n+2}$
which map to $X_i\otimes l_i^k\in(\sL_1 \otimes\sL_{n+1})^k$, determining $\mu^k_n(L)$.
\end{proof}

\subsection{The $k$-repeating twisted Whitney tower classification}\label{subsec:k-repeating-twisted-classification}
From Theorem~\ref{thm:k-repeating-mu} we have the following commutative triangle of groups and maps:

\begin{equation}\label{eqn:k-triangle}
    \xymatrix{
{\cT^k_n}^\iinfty \ar@{->>}[r]^{{R^k_n}^\iinfty} \ar@{->>}[rd]_{\eta^k_n} & {\W^k_n}^\iinfty \ar@{->>}[d]^{\mu^k_n}\\
& \sD^k_{n}
}
\tag{$\bigtriangledown^k$}
\end{equation}

It follows from Theorem~\ref{thm:twisted-three-quarters-classification} that the maps $\eta^k_n:{\cT_n^k}^\iinfty \to \sD^k_n$ are isomorphisms for $n\equiv 0,1,3\,\mod 4$, since they are restrictions of the maps $\eta_n$ to the $k$-repeating direct summands.
So the $k$-repeating Milnor invariants $\mu^k_n\colon {\W^k_n}^\iinfty\to \sD^k_n$ and the $k$-repeating twisted realization maps ${R^k_n}^\iinfty : {\cT_n^k}^\iinfty \to{\W_n^k}^\iinfty$ are isomorphisms for these orders.

Restricting to direct summands again, for $k\geq 4$ Proposition~\ref{prop:kerEta4j-2} yields isomorphisms $\Z_2 \otimes \sL^{\lfloor k/4\rfloor}_j\cong\Ker(\eta^k_{4j-2}:{\cT^{k^\iinfty}_{4j-2}}\to\sD^k_{4j-2})$
defined by sending $1\otimes J $ to $ \iinfty-\!\!\!\!-\!\!-\!\!\!\!<^{\,J}_{\,J}\,\,\in{\cT^{k^\iinfty}_{4j-2}}$ 
for rooted trees $J\in\sL^{\lfloor k/4\rfloor}_j$ of order $j-1$, where $\lfloor \,\cdot\,\rfloor$ denotes the floor function.

For $\sK^{k^\iinfty}_{4j-2}$ denoting the kernel of $\mu^k_{4j-2}: {\W^{k^\iinfty}_{4j-2}} \sra \sD^k_{4j-2}$, it follows from the $k$-repeating triangle diagram \eqref{eqn:k-triangle} %and Proposition~\ref{prop:kerEta4j-2} above 
that mapping $1\otimes J$ to 
$R^{k^\iinfty}_{4j-2}( \iinfty-\!\!\!\!\!-\!\!\!\!<^{\,J}_{\,J}\,\,)$ induces a surjection $\alpha_j^{k^\iinfty}: \Z_2 \otimes \sL^{\lfloor k/4\rfloor}_j
\sra \sK^{k^\iinfty}_{4j-2}$, for all $j\geq 1$.
Denoting by $\overline{\alpha^{k^\iinfty}_{j}}$ the induced isomorphism on $(\mathbb Z_2\otimes {\sf L}^{\lfloor k/4\rfloor}_{j})/\Ker \alpha_j^{k^\iinfty}$,
for $k\geq 4$ the \emph{$k$-repeating higher-order Arf invariants} are defined by
$$
\Arf^k_{j}:=(\overline{\alpha_j^{k^\iinfty}})^{-1}:\sK^\iinfty_{4j-2}\to(\mathbb Z_2\otimes {\sf L}^{\lfloor k/4\rfloor}_{j})/\Ker \alpha_j^{k^\iinfty}.
$$

We have the following $k$-repeating analogue of Conjecture~\ref{conj:Arf-j}:
\begin{conj}\label{conj:k-rep-Arf-j} 
$\Arf^k_{j}$ is an isomorphism for all $k$ and $j$.
\end{conj}

Regardless of the size of $\Ker\alpha_j^{k^\iinfty}$, we have the following $k$-repeating analogue of Corollary~\ref{cor:intro-mu-arf-classify-twisted}:
\begin{cor}\label{cor:k-repeating-mu-arf-classify-k-repeating-twisted} 
The groups ${{\sf W}^k_n}^\iinfty$ are classified by $k$-repeating Milnor invariants $\mu^k_n$ and, in addition, $k$-repeating higher-order Arf invariants $\Arf^k_j$ for $n=4j-2$ and $k\geq 4$.

In particular, a link bounds a $k$-repeating twisted Whitney tower of order $n+1$ if and only if its $k$-repeating Milnor invariants and $k$-repeating higher-order Arf invariants vanish up to order $n$.
$\hfill\square$
\end{cor}
This completes the proof of Theorem~\ref{thm:twisted-self-Ck-equals-Milnor-Arf-equals-tower}.

%%%%%%%%%%%%%%%%%%%%%%%%%%%%%%%%%%%%%%%%%%%%%%%%%%%%%%%%%%%%%%%%%%%

\end{document}